\documentclass[12pt,a4paper,oneside,notitlepage]{amsart}
\usepackage{amssymb}
\usepackage[arrow,matrix,curve]{xy}
\def\xyma{\xymatrix@M.7em}
\usepackage[centertags]{amsmath}
\usepackage[T1]{fontenc}
\usepackage[latin1]{inputenc}
\usepackage{amsfonts}
\usepackage{amscd}
\usepackage{euscript}
\usepackage{amsmath}

\numberwithin{equation}{section}
\newtheorem{theorem}{Theorem}
\newtheorem{rem}[theorem]{Remark}

\newtheorem{defi}[theorem]{Definition}
\newtheorem{prop}[theorem]{Proposition}
\newtheorem{lemma}[theorem]{Lemma}
\newtheorem{cor}[theorem]{Corollary}

\newcommand{\mr}[1]{ \stackrel{#1}{\longrightarrow} }

\newcommand{\ml}[1]{ \stackrel{#1}{\longleftarrow} }

\usepackage{amsfonts,amssymb}

\newcommand{\nc}{\newcommand}

\nc{\Ker}[1]{\mbox{Ker$(#1)$}}

\nc{\impl}{\Rightarrow}

\makeindex
\usepackage{amssymb}
\usepackage{amsmath}
\usepackage{stmaryrd}

\newdir{>>}{{}*!/3.5pt/:(1,-.2)@^{>}*!/3.5pt/:(1,+.2)@_{>}*!/7pt/:(1,-.2)@^{>}*!/7pt/:(1,+.2)@_{>}}
\newdir{ >>}{{}*!/8pt/@{|}*!/3.5pt/:(1,-.2)@^{>}*!/3.5pt/:(1,+.2)@_{>}}
\newdir{ |>}{{}*!/-3.5pt/@{|}*!/-8pt/:(1,-.2)@^{>}*!/-8pt/:(1,+.2)@_{>}}
\newdir{ >}{{}*!/-8pt/@{>}}
\newdir{>}{{}*:(1,-.2)@^{>}*:(1,+.2)@_{>}}
\newdir{<}{{}*:(1,+.2)@^{<}*:(1,-.2)@_{<}}

\begin{document}
\newcommand{\N}{\noindent}
\newcommand{\C}{\mbox{$\mathbb C$}}
\newcommand{\REF}[1]{(\ref{#1})}

\title{Internal object actions in 
%finitely cocomplete 
homological categories}

%\date{today}
\author{Manfred Hartl \& Bruno Loiseau}

\address{Universit\'e de Lille - Nord de France,  Laboratoire de Math\'ematiques et de leurs Applications de Valenciennes and FR CNRS 2956, UVHC, Le Mont Houy, 
F-59313 Valenciennes Cedex 9, France.}
\email{manfred.hartl@univ-valenciennes.fr}
\email{bruno.loiseau@lille.iufm.fr}

\maketitle

\begin{abstract}
Let $G$ and $A$ be objects of a finitely cocomplete homological category $\mathbb C$.
We define a notion of an (internal) action of $G$ of $A$ which is functorially equivalent with a point in $\mathbb C$ over $G$, i.e. a split extension in $\mathbb C$ with kernel $A$ and cokernel $G$. This notion and its study are based on a preliminary investigation of cross-effects of functors in a general categorical context. These also allow us to define higher categorical commutators.
% (which, however, will be further studied elsewhere).
We show that any proper subobject of an object $E$ (i.e., a kernel of some map on $E$ in $\mathbb C$) admits a ``conjugation'' action of $E$, generalizing the conjugation action of $E$ on itself defined by Bourn and Janelidze. If $\mathbb C$ is semi-abelian, we show that for subobjects $X$, $Y$ of some object $A$, $X$ is proper in the supremum of $X$ and $Y$ if and only if $X$ is stable under the restriction to $Y$ of the conjugation action of $A$ on itself. This amounts to an elementary proof of Bourn and Janelidze's functorial equivalence between points over $G$ in $\mathbb C$ and algebras over a certain monad $\mathbb T_G$ on $\mathbb C$. The two axioms of such an algebra can be replaced by three others, in terms of cross-effects, two of which generalize the usual properties of an action of one group on another.

%Moreover, our criterion allows to cut the axioms of such an algebra in three pieces two of which have a nice interpretation in the case where $\mathbb C$ is the category of groups, in terms of the usual notion of an action of one group on another.

\end{abstract}
\vspace{5mm}

\vspace{1cm}

For two objects $G$ and $A$ of a category $\mathbb C$, an \textit{action} of $G$ on $A$ morally should be some minimal data equivalent to a split extension with kernel $A$ and cokernel $G$ (also called a point of $\mathbb C$ over $G$), which then may be called the semi-direct product of $A$ and $G$ along the given data. For  semi-abelian categories $\mathbb C$, an elegant solution of the problem to exhibit such minimal data was first provided by Bourn and Janelidze in \cite{BJ}, by defining an action to be an algebra over a certain monad $\mathbb T_G$ on $\mathbb C$. This notion then was further studied by various authors, in particular in terms of monoidal categories by Borceux, Janelidze and Kelly in \cite{BJK}. 

In this paper, we define a notion of action of $G$ on $A$ in terms of even ``smaller'' data then a
$\mathbb T_G$-algebra, namely as being a map $(A|G)\to A$ satisfying a certain property, where $(A|G)$ is the kernel of the canonical map from the sum $A+G$ to the product $A\times G$. An action in this sense is functorially equivalent to a point over $G$  even in a non-exact context, namely in finitely cocomplete homological categories $\mathbb C$ (which are of special interest to the authors, as, for example, certain categories of filtered groups are of this type, see the example at the end of section 3). Moreover, our notion allows to directly generalize certain basic concepts from groups and Lie-algebras, such as a conjugation action of an object on any proper subobject, and to formalize the fact that the semi-direct product along some action can be viewed as its universal transformation into a conjugation action. This leads to a very useful characterization of proper subobjects in semi-abelian categories, namely as those being stable under the conjugation action (this result was recently  obtained independantly by Mantovani and Metere in \cite{MM}); and this property characterizes semi-abelian categories among finitely cocomplete homological ones. More generally, we define a notion of one subobject normalizing another one, in terms of the conjugation action, which is equivalent to the latter being proper in the supremum of both when $\mathbb C$ is semi-abelian (Theorem \ref{propercrit}). 

These facts lead to many applications: e.g., based on a detailed comparison of our actions with $\mathbb T_G$-algebras, they allow to reprove the equivalence between these algebras and points over $G$ when $\mathbb C$ is semi-abelian, in an elementary way without using Beck's criterion. More applications are given in our forthcoming study of extensions with non-abelian kernel and of higher commutators of subobjects as introduced in this paper (Definition \ref{comdef}). It is based on the notion of \textit{cross-effect} of a functor which is fundamental in algebraic topology; it was introduced, for functors between abelian categories, by Eilenberg and MacLane \cite{EML}, and adapted to functors with values in the category of groups by Baues and Pirashvili \cite{BP}, see also \cite{HV} for further properties. 
In section 1 of this paper, we define cross-effects in a general categorical context and exibit some elementary properties, which constitute our main tools: in fact, the object $(A|G)$ can be viewed as the second cross-effect of the identity functor of $\mathbb C$. In particular, we use the cross-effect  machinery to cut the axioms of a  
$\mathbb T_G$-algebra in three pieces, two of which again look like associativity conditions, on the terms $((A|G)|A)$ and $((A|G)|G)$, and have nice interpretations in the category of groups: the first one says that $G$ acts by endomorphisms of $A$, and the second then expresses the usual associativity condition for the action of $G$ on the underlying set of $A$. Thus the third condition  is void in the category of groups, but probably not in general. It involves a tripel cross-effect and actually looks quite odd; 
we did not yet succeed to give it a more convenient form.\vspace{1cm}

%\end{document}

\begin{center}\sc Basic conventions
\end{center}\medskip

Throughout this paper, $\mathbb C$ denotes a finitely cocomplete homological category, cf.\ \cite{BB}.
When  merely working in a pointed category with finite sums, we denote the canonical inclusion $X_k\rightarrow X_1+\ldots +X_n$ by $i_{X_k}$ or by $i_k$, and its canonical retraction by $r_{X_k}$ or by $r_k$; dually, when working in  a pointed category with finite products, we denote the canonical projection $X_1\times \ldots \times X_n\rightarrow X_k$ by $\pi_k$ and its canonical section by $\sigma_k$. When the category is pointed and has finite products and finite sums, the canonical map $ X_1+\ldots +X_n\rightarrow X_1\times \ldots \times X_n$ is denoted by $b_{X_1\ldots X_n}$ or simply by $b$.

%\end{document}

\section{Cross effects of functors}

The fundamental concept of cross-effect  of a functor between abelian categories is due to Eilenberg and MacLane \cite{EML}; it was adapted to functors with values in the category of groups in \cite{BP} and further studied  in \cite{HV}. 
In this paper and subsequent work we will tend to show, however, that the concept of cross-effect provides a powerful in the theory of homological and semi-abelian categories, so we keep our hypothesis' as general as possible. We here content ourselves to prove just some elementary properties of cross-effects in this framework; a more thorough investigation is carried out in \cite{CCC}.\medskip

Let $F: {\mathbb D} \to {\mathbb E}$ be a functor where 
${\mathbb D}$ is a  pointed category  with finite sums and 
${\mathbb E}$ is a  pointed finitely complete category.

\begin{defi} For any nonzero natural number $n$, the $n$-th cross effect of $F$ is defined to be the functor $cr_n(F):\mathbb D^{\times n}\rightarrow \mathbb E$ inductively defined  by:
\[ cr_1(F)(X)=\mbox{Ker}(F(0)\colon F(X)\rightarrow F(0))\,;\]
\[cr_2(F)(X,Y)=\mbox{Ker \big($(F(r_X),F(r_Y))^t\colon F(X+ Y)\rightarrow F(X)\times F(Y)\big)$}\,;\]
\[ cr_n(F)(X_1,X_2,\ldots X_n)=cr_2\big(cr_{n-1}(F)(-,X_3,\ldots,X_n)\big)(X_1,X_2)\,, \quad\mbox{for $n\geq3$}\,.\]
The definition of $cr_1$ and $cr_2$ (hence of $cr_n$) on morphisms is  obvious.
\end{defi}

One often abbreviates $cr_n(F)(X_1,X_2,\ldots X_n)=F(X_1|\ldots|X_n) $.  For the identity functor $Id_{\mathbb C}$ of $\mathbb C$ and objects $X_1,\ldots,X_n$ in $\mathbb C$  we write
$(X_1|\ldots|X_n) =  Id_{\mathbb C}(X_1|\ldots|X_n)$, so that for instance, by the very definition, $(X|Y|Z)$ is the kernel of the canonical map from $(X+Y|Z)$ to $(X|Z)\times (Y|Z)$.

\begin{prop}\label{elcrprops} The functors $cr_n$ have the following properties:
\begin{enumerate}

\item There is a natural injection
\[ \xymatrix{
 \iota_{X_1,\ldots ,X_n}^F \colon cr_n(F)(X_1,\ldots,X_n) \ar@{{ >}->}[r] & F(X_1+\ldots+ X_n) }\]

\item The functor $cr_n$ is multireduced, i.e.\ $cr_n(F)(X_1,\ldots,X_n) = 0$ if $X_k=0$ for some $k=1,\ldots,n$.

\item The bifunctor $cr_2(F)$ is symmetric.

\end{enumerate}

\end{prop}

In fact, we show in \cite{CCC} that the functor $cr_n(F)$ is symmetric for all $n$, but we do not need this here.

When no confusion can arise we often just write $\iota$ instead of $\iota_{X_1,\ldots ,X_n}^F$.

\begin{proof} Assertion (3) is obvious, the other two are easily proved by induction.
For (2), just observe that if $X=0$, also $cr_2(X,Y)=0$ since the map $(F(r_X),F(r_Y))^t \colon F(0+Y) \to F(0)\times F(Y)$ admits the second projection followed by $F(i_{Y})$ as a retraction.
\end{proof}

The following facts are  key tools in handling cross effects.

\begin{prop}\label{cr-reg-epi} Suppose in  addition  that $\mathbb E$ is homological and that $F$ preserves regular epimorphisms. Then for all objects $A$ in $\mathbb D$ the functor $F(A|-):\mathbb D \to \mathbb E$ also preserves regular epimorphisms.
\end{prop} 

\begin{proof} Let $\xymatrix{ f\colon X \ar@{>>}[r]&Y}$ be a regular epimorphism.  Consider the following commutatif diagram where  $\xymatrix{k\colon {\rm Ker}(F(f)) \ar@{{ |>}->}[r] &F(X)}$ is a kernel of $F(f)$ and $\alpha$ is induced by $(F(r_X),F(r_Y))^t$.
:
\[\xymatrix{ & & F(A|X) \ar[d]^{\iota} \ar[r]^-{F(1|f)} & F(A|Y)\ar[d]^{\iota}   & \\
& {\rm Ker}(F(1+f)) \ar[d]^{\alpha} \ar@{^(->}[r]^j & F(A+X) \ar[r]^-{F(1+f)}  \ar[d]^{(Fr_A,Fr_X)^t} & F(A+Y) \ar[d]^{(Fr_A,Fr_Y)^t} \ar[r] & 0 \\
0\ar[r] & {\rm Ker}(F(f)) \ar[r]^-{(0,k)^t} & F(A)\times F(X) \ar[r]^-{1\times F(f)} & F(A)\times F(Y)& \rule[-3mm]{0mm}{0mm}
}\]
The columns are exact by definition of the cross-effect, and the rows are exact, too; for the middle row this follows from the hypothesis on $F$ since     $1+f$ is a regular epimorphism. Thus the snake lemma provides an exact sequence 
\[ \xymatrix{ F(A|X) \ar[r]^-{F(1|f)} & F(A|Y) \ar[r] & {\rm Coker}(\alpha) 
}\]
We claim that ${\rm Coker}(\alpha) =0$: in fact, 
the map $F(i_X)k\colon {\rm Ker}(F(f)) \to F(A+X)$ factors through $j$ and thus provides a section $s$ of $\alpha$, indeed:
$F(1+f)F(i_Y)k = F(i_Y)F(f)k=0$, and
\[(0,k)^t \alpha s = (Fr_A,Fr_X)^t j s = (Fr_A,Fr_X)^t F(i_X)k = (0,1)^t k = (0,k)^t\,
\]
whence $\alpha s =1$ since $(0,k)^t$ is monic.

\end{proof}

\begin{prop}\label{crsdp} Suppose  in  addition that ${\mathbb D}$ is protomodular, $\mathbb E$ is homological and that $F$ preserves regular epimorphisms.
Let $f:X\to Y$ be a map in ${\mathbb D}$ with splitting $s:Y\to X$, i.e.\ such that $fs=1$. Let $k:K \to X$ be a kernel of $f$ and let $A\in Ob({\mathbb D})$. Then the map
%\[<\iota'   , F(1|k), F(1|s) >\hspace{2mm}:\hspace{1mm}
%F(A|K|Y) + F(A|K) + F(A|Y) \to F(A|X)\]
\[<\iota'   , F(k|1), F(s|1) >\hspace{2mm}:\hspace{1mm}
F(K|Y|A) + F(K|A) + F(Y|A) \to F(X|A)\]
is a regular epimorphim where $\iota'$ is the composite map
$$\xymatrix{F(K|Y|A) = F(-|A)(K|Y) \ar[rr]^-{\iota}&&F(K+ Y|A) \ar[rr]^{F(<k,s>|1)}&&F(X|A)}$$%
where more precisely $\iota = \iota_{K,Y}^{F(-|A)}$.
%F(A|-):\mathbb D \to \mathbb E$.

\end{prop}

\begin{proof} Consider the following commutative diagram whose right-hand square is a pullback:
\[\xymatrix{
0\ar[r] & F(K|Y|A) \ar@{=}[d] \ar[r]^-{(\iota,0)^t} & P\ar[d]^{p_1} \ar[r]^-{p_2} & F(K|A) + F(Y|A)  \ar[d]^{(r_ 1,r_2)^t}&  \\
0\ar[r] & F(K|Y|A) \ar[r]^{\iota} & F(K+Y|A) \ar[r]^-{r} & F(K|A) \times  F(Y|A)
}\]
Here $r=(F(r_K|1), F(r_Y|1))^t$, $r_1=r_{F(K|A)}$ and $r_2=r_{F(Y|A)}$. The bottom row is exact by definition of the cross-effect, hence the top row is also exact. Let $\sigma\colon F(K|A) + F(Y|A) \to P$ be the map such that $p_1\sigma = <F(i_K|1),F(i_Y|1)>$ and $p_2\sigma = 1$. Then $\sigma$ is a splitting of $p_2$, whence by protomodularity  of $\mathbb{E}$ the map
\[ < (\iota,0)^t , \sigma > \hspace{1mm} \colon \hspace{1mm}F(K|Y|A) + F(K|A) + F(Y|A) \longrightarrow P\]
is a regular epimorphism. We then successively conclude that the following maps also are regular epimorphisms:
\begin{itemize}

\item $(r_1,r_2)^t$  by protomodularity of $\mathbb{E}$;

\item $p_1$   by regularity of $\mathbb{E}$;

\item $<k,s>$ by protomodularity of $\mathbb{D}$;

\item $F(<k,s>|1)$ by Proposition \ref{cr-reg-epi};

\item the composite map $F(<k,s>|1) p_1 < (\iota,0)^t , \sigma > $.

\end{itemize}

But 
\begin{eqnarray*}
F(<k,s>|1) p_1 < (\iota,0)^t , \sigma >  &=&
F(<k,s>|1)<\iota, F(i_K|1),F(i_Y|1)> \\
&=& 
<\iota'   , F(k|1), F(s|1) >\,,
\end{eqnarray*}
whence the assertion.
\end{proof}

\section{\label{genprop} General properties of internal object actions}

Recall that the category of (internal) \textit{actions} of an object $G$ of $\mathbb C$ can be equivalently defined as being 
\begin{itemize}
\item The category of split extensions of $G$, whose objects are short split exact sequences $$\xymatrix{0\ar[r]&A\ar[r]&E\ar@<-2pt>[r]_q&G\ar@/_/@<-2pt>[l]_s\ar[r]&0}$$ ( the sequence is exact and $s$ is a spitting of $q$), and such an object of this category is called an action of $G$ on $A$. A map from 
$\xymatrix{0\ar[r]&A\ar[r]&E\ar@<-2pt>[r]_q&G\ar@/_/@<-2pt>[l]_s\ar[r]&0}$ to $\xymatrix{0\ar[r]&A'\ar[r]&E'\ar@<-2pt>[r]_{q'}&G\ar@/_/@<-2pt>[l]_{s'}\ar[r]&0}$ 
being the data of two maps $a:A\rightarrow A'$ and $b:E\rightarrow E'$ making the following diagram commute:
$$\xymatrix{
0\ar[r]&A\ar[d]_a\ar[r]&E\ar[d]_b\ar@<-2pt>[r]_q&G\ar@/_/@<-2pt>[l]_s\ar[r]\ar@{=}[d]&0\\
0\ar[r]&A'\ar[r]&E'\ar@<-2pt>[r]_{q'}&G\ar@/_/@<-2pt>[l]_{s'}\ar[r]&0}$$
\item The category $\mbox{Pt}_G(\mathbb C)$ of $G$-points of $\mathbb C$, i.e. the category $(\mathbb C/G)\backslash(1_G)$, which can described as the category of  objects $E$ of $\mathbb C$ together with a map $q:E\rightarrow G$ and a section $s:G\rightarrow E$ of $q$. A map $b$ between two such objects $(E,q, s)$ and $(E',q',s')$ is a map $b: E\rightarrow E'$ making the following diagram commute: 
$$\xymatrix{
E\ar[d]_b\ar@<-2pt>[r]_q&G\ar@/_/@<-2pt>[l]_s\ar@{=}[d]\\
E'\ar@<-2pt>[r]_{q'}&G\ar@/_/@<-2pt>[l]_{s'}}$$

%\item or a map $m:(A|G) \to A$  satisfying the following property: let $A \mr{i_A} A+ G \ml{i_G} G$ denote the defining injections of the coproduct,  let $q:A+ G \to Q$ be  a coequalizer of the maps $\iota,i_Am: (A|M) \to A+ G$, and let $i_A^{\prime}=q i_A$ and $i_G^{\prime}=q i_G$. Then to be an action $m$ is required to be such that $ i_A^{\prime}$ is a monomorphism, in which case we get a split short exact sequence 
%\[0\to A \mr{i_A^{\prime}} Q \to G \to 0\]
%with splitting $s=qi_G^{\prime}$; we call $Q$ the (standard) semi-direct product of $A$ and $G$ along $m$ and denote it by $Q=A\rtimes_m G$. It turns out that every split extension of $A$ by $G$ is isomorphic to the standard one above, in fact, the above construction of a semi-direct product furnishes    an equivalence between the category of actions and the category of split extensions in \C\ \textit{(this approach is developped in a forthcoming paper of Bruno and myself)};

\item if $\mathbb C$ is moreover exact, hence semi abelian, these two obviously equivalent categories are also equivalent to the category of Eilenberg-Moore algebras over the monad $\mathbb T_G:\C \to \C$ sending an object $A$ to the kernel $T_G(A)$ of the natural retraction $A+ G \to G$, according to Bourn and Janelidze, see \cite{BJ}).

\end{itemize}

The monadic aspect may of course be studied in any pointed finitely complete category with finite sums, as in \cite{BJK} where $T_G(A)$ is considered as a subobject of $G+A$ rather than of $A+G$, and is  denoted by $G\flat A$.

Recall that $\mbox{Pt}_G(\mathbb C)$ is a subcategory of a more general category $\mbox{Pt}(\mathbb C)$, whose objects are points (on variable objects $G$), and a map between two points 
$(G,E,q, s)$ and $(G',E',q',s')$ is a pair of maps $a:G\to G'$ and $b: E\rightarrow E'$ making the following diagram commute: 
$$\xymatrix{
E\ar[d]_b\ar@<-2pt>[r]_q&G\ar@/_/@<-2pt>[l]_s\ar[d]^a\\
E'\ar@<-2pt>[r]_{q'}&G'\ar@/_/@<-2pt>[l]_{s'}}$$

We introduce another way to define a category of $G$-actions, using the (second) cross-effect of the identity functor of $\mathbb C$; this new category is always equivalent to the category of points, hence to the category of algebras described above when $\mathbb C$ is exact; but we also will exhibit the link between our definition and this category of algebras in the non-exact case. We start by relating the cross-effect of the identity functor to the functor $T_G$.

%In the sequel, we will mostly refer to the second definition; the first, however,  crucially appears in certain proofs.
\begin{lemma}\label{inclcrosseff}Let us denote by $k$ (or $k_{A,G}$) the inclusion of $T_G(A)$ in $A+G$. Then there exists an inclusion $j$ (or $j_{A,G}$) of $(A|G)$ in $T_G(A)$, yielding a split short exact sequence
$$\xymatrix{
0\ar[r]&(A|G)\ar[r]^j&T_G(A)\ar[r]_{r_A.k}&A\ar@/_1pc/[l]_{\eta_A}\ar[r]&0
}$$
where $\eta$ is the unit of the monad $\mathbb T_G$, i.e. $\eta_A$ is the unique map $A\rightarrow T_G(A)$ such that $k.\eta_A=i_A$, which exists and is unique since  $r_G.i_A=0$. \end{lemma}

\begin{proof}Since $r_G.\iota=\pi_G.b.\iota=0$, we get a unique $j:(A|G)\rightarrow T_G(A)$ such that $k.j=\iota$. It remains to show that $j$ is the kernel of $r_A.k$. First, $r_A.k.j=r_A\iota=\pi_A.b.\iota=0$. Second, let $f:X\rightarrow T_G(A)$ be such that $r_A.k.f=0$, then since one also obviously has $r_G.k.f=0$, one has $b.k.f=0$, so $kf$ factorizes through $(A|G)$, say $kf=\iota\bar{f}=k.j.\bar{f}$, hence $f=j.\bar{f}$ and we get a factorization of $f$ by $(A|G)$, which is unique since $j$ is a monomorphism.\end{proof}

\begin{cor}\label{restract}A map $\xi:T_G(A)\rightarrow A$ satisfying the unit axiom (i.e., $\xi \eta_A=1$) is uniquely determined by its restriction $\psi$ (along $j$) to $(A|G)$.\end{cor}

\begin{proof}Consider two maps $\xi$ and $\xi' : T_G(A)\rightarrow A$ satisfying this unit axiom, i.e. $\xi.\eta_A:\xi'.\eta_A=1_A$ having same restriction $\psi$ to $(A|G)$, i.e. $\psi=\xi.j=\xi'.j$. Then $\xi=\xi'$ since by Lemma \ref{inclcrosseff} and by protomodularity the pair $(\eta,j)$ is (strongly) epimorphic.\end{proof}

Corollary \ref{restract} explains why we will characterize our actions in terms of maps $(A|G)\rightarrow A$ rather than $T_G(A)\rightarrow A$.

Recall that when $\xi:T_G(A)\rightarrow A$ is an action of $G$ on $A$ (in the sense of \cite{BJK}), the semi-direct product of $A$ and $G$ along $\xi$  is  the coequalizer of $<k,i_G>$ and $\xi + 1:T_G(A)+G \to A+G$  (as defined in \cite{BJK}, but using our terminology and putting the $G$'s on the right). This definition arises naturally from the general theory of monads, but it is clear that this coequalizer is also the coequalizer of $k$ and $i_A\xi$. Here again, we show that the knowledge of the restriction $\psi:(A|G)\rightarrow G$ of $\xi$ to $(A|G)$ along $j:(A|G)\rightarrow T_G(A)$ suffices to determine this coequalizer:

\begin{prop}\label{coequalizer}Consider a map $\xi:T_G(A)\rightarrow A$ satisfying the unit axiom of a $\mathbb T_G$-algebra, and consider $\psi=\xi.j:(A|G)\rightarrow A$. Then the coequalizer of $k$ and $i_A\xi$ (hence the semi-direct product of $A$ and $G$ along $\xi$, if moreover $\xi$ is a $\mathbb T_G$-algebra, considering the observation here above) is also the coequaliser of $\iota$ (= $k.j$) and $i_A.\psi$.
\end{prop}

\begin{proof}We have to show that for any map $h:A+G\rightarrow X$ one has : $hk=hi_A\xi$ if and only if $hkj=hi_A\psi=h i_A\xi.j$. And of course only the sufficient condition must be proved; but it follows immediately from the fact that the pair $(j,\eta)$ is (strongly) epimorphic by Lemma \ref{inclcrosseff}.
\end{proof}

The following proposition underlines the key role of the object $(A|G)$ :

\begin{prop}\label{coeq}Let  $G$ and $A$ be objects of $\mathbb C$. Consider a morphism $\psi:(A|G)\rightarrow A$, and let $q_{\psi}:A+ G\rightarrow Q_{\psi}$ be the coequalizer of $\iota$ and $i_1.\psi$. Let $l_{\psi}$ be the composite $q_{\psi}.i_A$. Then:

1) $r_G$ coequalizes $\iota_{A,G}$ and $r_A.\psi$, giving rise to a unique extension $p_{\psi}:Q_{\psi}\rightarrow G$, such that $p_{\psi}q_{\psi}=r_G$;

2) The morphism $s_{\psi}=q_{\psi}.i_G: G\rightarrow Q_{\psi}$ is a section of $p_{\psi}$;

3) The sequence $\xymatrix{A\ar^{l_{\psi}}[r]&Q_{\psi}\ar^{p_{\psi}}[r]&G}$ is exact, hence the sequence $$\xymatrix{0\ar[r]&A\ar[r]^{l_{\psi}}&Q_{\psi}\ar[r]_{p_{\psi}}&G\ar@/_1pc/[l]_{s_{\psi}}\ar[r]&0}$$ is split short exact if and only if $q_{\psi}.i_A$ is a monomorphism. %When these equivalent conditions are satisfied, we denote $Q_{\psi}$ by $H\rtimes_{\psi}G$ and we denote the monomorphism  $q_{\psi}.i_H$ by $j_{\psi}$.
\end{prop}

\begin{proof}

1. One has:  $r_G.\iota = \pi_G.b_{A,G}.\iota=0$ and $r_G.i_A.\psi=0.\psi=0$.

2) $p_{\psi}.s_{\psi}=p_{\psi}q_{\psi}.i_G=r_g.i_G=1_G$.

3) Let $k':\mbox{Ker}(q_{\psi})\rightarrow A+G$ be the kernel of $q_{\psi}$. One has: $r_G.k'=p_{\psi}. q_{\psi}.k=0$, hence $\mbox{Ker}(q_{\psi})$ is a subobject of $T_G(A)$. So by Noether's first isomorphism theorem, it is proper, and $$\frac{\frac{A+ G}{\mbox{\tiny Ker}(q_{\psi})}}{\frac{T_G(A)}{\mbox{\tiny Ker}(q_{\psi})}}=\frac{A+ G}{T_G(A)}=G$$ 
meaning more precisely that we have the following diagram, where all sequences are short exact: 

$$\xymatrix{
&0\ar[d]&0\ar[d]&0\ar[d]\\
0\ar[r]&\mbox{Ker}(q_{\psi})\ar@{=}[d]\ar[r]&T_G(A)\ar[d]^k\ar[r]^{q'}&\mbox{Ker}(p_{\psi})\ar[r]\ar[d]^{k''}&0\\
0\ar[r]&\mbox{Ker}(q_{\psi})\ar[d]\ar[r]^{k'}&A+G\ar[d]^{r_G}\ar[r]^{q_{\psi}}&Q_{\psi}\ar[r]\ar[d]^{p_{\psi}}&0\\
0\ar[r]&0\ar[d]\ar[r]&G\ar[d]\ar@{=}[r]&G\ar[d]\ar[r]&0\\
&0&0&0}$$

Now one has: $p_{\psi}.q_{\psi}.i_A=r_G.i_A=0$, so $q_{\psi}.i_A$ factorizes by $\mbox{Ker}(p_{\psi})$, say $q_{\psi}.i_A=k''.v$ for some unique $v:A\rightarrow\mbox{Ker}(p_{\psi})$. So it remains to show that $v$ is a regular epimorphism to get the result.

By Lemma \ref{inclcrosseff} and by protomodularity of $\mathbb C$, the map $<j,\eta_A>:(A|G)+A\rightarrow T_G(A)$ is a regular epimorphism, and so is $q'$, hence so is $q'.<j,\eta_A>$. It suffices then to  show that $q'.<i,\eta_A>=v.<\psi,1_A>$, which will show that $v$ is a regular epimorphism as required. Of course, it suffices to show that $k''.q'.<j,\eta_A>=k''v.<\psi,1_A>$, or equivalently that $q_{\psi}.k.<j,\eta_A>=q_{\psi}.i_A<\psi,1_A>$. By preceeding by the inclusions of $(A|G)$ and $A$ in their sum, this amounts to show that $q_{\psi}.k.j=q_{\psi}.i_A.\psi$ and that $q_{\psi}.k.\eta_A=q_{\psi}.i_A.1_A=q_{\psi}.i_A$. First, $q_{\psi}.k.i=q_{\psi}.\iota=q_{\psi}.i_A.\psi$ by the very definition of $q_{\psi}$. Second, $k.\eta_A=i_A$, which gives the result.

\end{proof}

\noindent{\bf Example: }Let $\mathbb C$ be the category of groups. In order to simplify notations here, we will feel free to consider the inclusions of $A$ and $G$ in their sum as set inclusions, so that for instance, for $a\in A$ and $g\in G$ the notation $[g,a]=gag^{-1}a^{-1}$ denotes unambiguously the commutator of $g$ and $a$, seen as elements of $A+G$.  It is well known (see for instance \cite{MKS}) that for two groups $A, G$, the cross effect $(A|G)$ is freely generated by the elements $[g,a]$, $(g,a)\in G^*\times A^*$, where $X^*$ denotes the set of nonzero elements of the group $X$. More formally, for a set $S$ denote by $\mathcal F(S)$
the free group with basis $S$. Then the homomorphism $\mathcal F(G^*\times A^*)\to A+G$ sending a basis element $(g,a)$ to $[g,a]$ maps isomorphically onto $(A|G)$\footnote{Of course, it would seem more natural to present this isomorphism, obviously equivalently, as an isomorphism $\mathcal F(A^*\times G^*)\rightarrow (A|G)$. But for technical reasons which will appear clearly, $\mathcal F(G^*\times A^*)\rightarrow (A|G)$ is more convenient.}. 
%
%It is obvious that, for any $a\in A$ and $g\in G$ (hence in particular for nonzero such elements), one has $b_{A,G}([g,a])=0$, where $[g,a]$ denotes the commutator of $g$ and $a$, thus one gets an application $A^*\times G^*\rightarrow(A|G)$, which extends uniquely to a group homomorphism $\mathcal F(G^*\times A^*)\rightarrow (A|G)$\footnote{Of course, it would seem more natural to present this isomorphism, obviously equivalently, as an isomorphism $\mathcal F(A^*\times G^*)\rightarrow (A|G)$. But for technical reasons which will appear clearly, $\mathcal F(G^*\times A^*)\rightarrow (A|G)$ is more convenient.}. 
Indeed, the fact that this map maps surjectively onto $(A|G)$ can be easily seen as follows: any element of $A+G$ can be written as a product of the form $\left(\prod_{i=1}^n[g_i,a_i]^{z_i}\right).a.g$, where the $z_i$'s are integers, $g$ and the $g_i$'s are in $G$ and $a$ and the $a_i$'s are in $A$. Moreover, $b_{A,G}\left(\left(\prod_{i=1}^n[g_i,a_i]^{z_i}\right).a.g\right)=(a,g)$, showing the uniqueness of $a$ and $g$ and showing that the kernel of $b_{A,G}$ is the set of elements of $A+G$ having the form $\prod_{i=1}^n[g_i,a_i]^{z_i}$, hence that we get an epimorphism $\mathcal F(G^*\times A^*)\rightarrow (A|G)$. That this epimorphism is an isomorphism is less easy to show.

So a map $\psi:(A|G)\rightarrow A$ can be seen as a set-theoretic application $G^*\times A^*\rightarrow A$. The link with usual actions is the following. Let $\phi: G\times A\rightarrow A$ be an action in the usual sense. Then one also may consider $\phi': G^*\times A^*\rightarrow A$ defined by $\phi'(g,a)=\phi(g,a).a^{-1}$, and $\psi:\mathcal F(G^*\times A^*)=(A|G)\rightarrow A$ its extension as a group morphism. Then it is easy to see that the coequalizer of $i_A.\psi$ and $\iota$ is the canonical epimorphism from $A+G$ to the usual semi-direct product $A\rtimes_{\phi}G$. Conversely, if $\psi:\mathcal F(G^*\times A^*)=(A|G)\rightarrow A$ is such that $q_{\psi}.i_A$ (as constructed above) is a monomorphism, then it gives rise to a split extension $\xymatrix{0\ar[r]&A\ar[r]^{q_{\psi}.i_A}&Q_{\psi}\ar[r]_{p_{\psi}}&G\ar@/_/[l]_{s_{\psi}}\ar[r]&0}$, hence to an action $\phi:G\times A\rightarrow A$ in the usual sense. Then it can be shown that for nonzero $g\in G$ and $a\in A$ one has $\psi(g,a)=\phi(g,a).a^{-1}=\phi'(g,a)$. The resulting equivalence between maps $\phi$ or $\phi'$ and suitable homomorphisms $\psi$ as above motivates the following definition for any finitely cocomplete homological category:

\begin{defi}\label{action} An {\em action} (of an object $G$ on an object $A$) is the data of two objects $A$ and $G$ and of a map $\psi:(A|G)\rightarrow A$ such that $q_{\psi}.i_A$ is a monomorphism. A map of actions: $(A,G,\psi)\rightarrow(A',G',\psi')$ is an ordered pair $(a,g)$ where $a:A\rightarrow A'$ and $g:G\rightarrow G'$ are maps in $\mathbb C$ making the following diagram commute:
$$\xymatrix{
(A|G)\ar[r]^{(a|g)}\ar[d]_{\psi}&(A'|G')\ar[d]_{\psi'}\\
A\ar[r]^a&A'
}$$

If $(A,G,\psi)$ is an action, then we denote by $A\rtimes_{\psi}G$ the quotient $Q_{\psi}$ as defined above.
\end{defi}

This obviously defines a category which we denote by $\mbox{Act}({\mathbb C})$. It is equipped with a forgetful functor to $\mathbb C$, sending the object $(A,G,\psi)$ to the object $G$, and the map $(a,g)$ to $g$. 
This makes of $\mbox{Act}(\mathbb C)$ a fibration on $\mathbb C$ (as is easily shown, but also as a consequence of Proposition \ref{equivact} below). The fiber on $G$ of this fibration is denoted by $\mbox{Act}_G(\mathbb C)$.

As can be expected, this notion of semi-direct product satisfies a universal property (see Proposition \ref{PU} below for another version which more directly generalizes  the usual one in the category of groups):

\begin{prop}\label{firstunivcondact}
Let $C$ be an object of $\mathbb C$ and $f_A: A\rightarrow C$, $f_G:G\rightarrow C$ be two morphisms such that the following diagram commutes:
$$\xymatrix{(A|G)\ar[r]^{\iota_{A,G}}\ar[d]_{\psi}&A+ G\ar[d]^{<f_A,f_G>}\\
A\ar[r]_{f_A}&C}$$
Then there exists a unique $f:A\rtimes_{\psi}G\rightarrow C$ such that $f.l_{\psi}=f_A$ and $f.s_{\psi}=f_g$:

$$\xymatrix{
A\ar@/^1pc/[drr]^{f_A}\ar[dr]_{l_{\psi}}\\
&A\rtimes_{\psi}G\ar@{.>}[r]^{\exists !f}&C\\
G\ar[ur]^{s_{\psi}}\ar@/_1pc/[rru]_{f_G}}$$

\end{prop}

\begin{proof}
Consider the map $<f_A,f_G>:A+ G\rightarrow C$. One has: $<f_A,f_G>.\iota_{A,G}=f_A.\psi=<f_A,f_G>.i_A.\psi$. So, since $A\rtimes_{\psi}G$ is the coequalizer of $\iota_{A,G}$ and $i_A.\psi$, there exists a unique $f:A\rtimes_{\psi}G\rightarrow C$ such that $f.q_{\psi}=<f_A,f_G>$:
$$\xymatrix{
(A|G)\ar[r]^{\iota_{A,G}}\ar[d]_{\psi}&A+ G\ar[d]^{q_{\psi}}\ar@/^1pc/[dr]^{<f_A,f_G>}\\
A\ar[ru]^{i_A}\ar[r]_{l_{\psi}}&A\rtimes_{\psi}G\ar@{.>}[r]_{\exists !f}&C}$$
Then $f.l_{\psi}=f.q_{\psi}.i_A=<f_A,f_G>.i_A=f_A$, and $f.s_{\psi}=f.q_{\psi}.i_G=<f_G,i_G>.i_G=f_G$. And since $\mathbb C$ is protomodular, the family $(l_{\psi},s_{\psi})$ is (strongly) epimorphic, which ensures uniqueness of $f$ with the required property.
\end{proof}

We may now compare our category of actions to the category of Eilenberg-Moore algebras on $\mathbb T_G$, for fixed $G$:

\begin{prop} \label{compar}Any action $\psi:(A|G)\rightarrow A$ extends uniquely to an Eilenberg-Moore algebra $\xi:T_G(A)\rightarrow A$ ; this gives rise to a full and faithful functor $\Xi_G:\mbox{Act}_G(\mathbb C)\rightarrow\mathbb C^{\mathbb T_G}$. Moreover, $\Xi_G$ is ``injective on objects'' so that if $\mathbb X_G$ is the full subcategory of the objects of $\mathbb C^{\mathbb T_G}$ which are images by $\Xi_G$ of objects $\psi$ of  $\mbox{Act}_G(\mathbb C)$, then $\Xi_G$ is an isomorphism of categories between $\mbox{Act}_G(\mathbb C)$ and $\mathbb X_G$.% which moreover satisfies $\Xi_G.\Psi_G=\mathcal J_G$, where $\mathcal J_G$ is the comparison functor $\mbox{Pt}_G(\mathbb C)\rightarrow \mathbb C^{\mathbb T_G}$ of the adjunction.
\end{prop}
\begin{proof}
If $\psi$ is an action in our sense, then we may construct $q_{\psi}:A+G\rightarrow A\rtimes_{\psi}G$ the coequalizer of $\iota_{A,G}$ and $i_A.\psi$, and $p_{\psi}:A\rtimes_{\psi}G\rightarrow G$ like in Proposition \ref{coeq}. Then since $\psi$ is an action, $l_{\psi}=q_{\psi}.i_{\psi}$ is the kernel of $p_{\psi}$. 
%Then considering that $\iota=k.j:(A|G)\rightarrow T_G(A)\rightarrow A+G$, 
One has $p_{\psi}.q_{\psi}.k=r_G.k=0$, hence there exists a unique $\xi:T_G(A)\rightarrow A$ such that $q_{\psi}.k=l_{\psi}.\xi$. One has $\xi.j_{A,G}=\psi$, since $l_{\psi}\psi=q_{\psi}i_A\psi=q_{\psi}\iota=q_{\psi}kj=l_{\psi}\xi$ and $l_{\psi}$ is a monomorphism. We put $ \Xi_G(\psi)=(A,\xi)$. If $\psi':(A'|G)\rightarrow A'$ is another $G$-action, and $a:A\rightarrow A'$ a map in $\mathbb C$, then the diagram
$$\xymatrix{T_GA\ar[d]_{\xi}\ar[r]^{T_Ga}&T_GA'\ar[d]^{\xi'}\\
A\ar[r]_a&A'}$$commutes if and only if $a$ is a map of actions, again since $(\eta,j)$ is a (strongly) epimorphic pair. This shows that if one can ensure that $\xi$ and $\xi'$ are Eilenberg-Moore algebras over $\mathbb T_G$ and $a$ is a map of actions, then $a$ also is a morphism of algebras. We thus obtain a full and faithfull functor $\Xi_G$ : $\mbox{Act}_G(\mathbb C)\rightarrow\mathbb C^{\mathbb T_G}$, defined on objects as above, and on maps by $\Xi_G(a)=a$.
%Now if $\mathcal X = (\xymatrix{X\ar[r]_p&G\ar@/_/[l]_s})$ and $\mathcal X'=(\xymatrix{X'\ar[r]_{p'}&G\ar@/_/[l]_{s'}})$ are two $G$-points and $x:X\rightarrow X'$ is a morphism of $G$-points between them, then it is easy to see that $\Xi_G(\Psi_G(\mathcal X))$ is nothing but the pair $(R(\mathcal X), R(\epsilon_{\mathcal X}))$, where $R$ is the right adjoint of the adjunction between $\mbox{Pt}_G(\mathbb C)$ and $\mathbb C$ giving rise to the monad $\mathbb T_G$, i.e. $R$ is the kernel functor, and $\epsilon$ is its counit. And $\Xi_G(x)=R(x)$. But this is precisely the construction of the functor $\mathcal J_G$, so that $\Xi_G.\Psi_G=\mathcal J_G$, and in particular it takes values in $\mathbb C^{\mathbb T_G}$. Then $\Xi_G$ takes values in $\mathbb C^{\mathbb T_G}$ since $\Psi_G$ is an equivalence of categories.
Now it is easy to see that $\xi$ is nothing but the pair $(R(A\rtimes_{\psi}G), R(\epsilon_{A\rtimes_{\psi}G}))$, where $R$ is the right adjoint of the adjunction between $\mbox{Pt}_G(\mathbb C)$ and $\mathbb C$ giving rise to the monad $\mathbb T_G$, i.e. $R$ is the kernel functor, and $\epsilon$ is its counit. %And $\Xi_G(x)=R(x)$. 
But this is precisely the construction of the comparison functor $\mathcal J_G :\mbox{Pt}_G(\mathbb C)\rightarrow \mathbb C^{\mathbb T_G}$ of the adjunction, and in particular it takes values in $\mathbb C^{\mathbb T_G}$.
Note also that $\xi.j=\psi$ (since it is true, followed by the mono $l_{\psi}$), showing that $\Xi_G$ is ``injective on objects''.

\end{proof}

\begin{prop}\label{equivact}
If $\psi:(A|G)\rightarrow A$ is an action, then the construction of  $A\rtimes_{\psi}G$ coincides with the one of $G\ltimes (A,\Xi_G(\psi))$ in \cite{BJK}, hence is functorial. The comparison adjunction $(F',G',\eta',\epsilon') : \mathbb C^{\mathbb T_G}\rightarrow \mbox{Pt}_G(\mathbb C)$  in \cite{BJK} (where $F'$ essentially is the semi-direct product functor and $G'$ is the functor $\mathcal J_G$ above) restricts to an equivalence between $\mathbb X_G$ and $\mbox{Pt}_G(\mathbb C)$. Hence precomposition with $\Xi_G$ provides an equivalence of categories between  $\mbox{Act}(\mathbb C)$ and $\mbox{Pt}(\mathbb C)$, which is compatible with the forgetful functors, so that $\mbox{Act}(\mathcal C)$ is a fibration whose fibres are the categories $\mbox{Act}_G(\mathbb C)$, and the ``inverse'' functors $\Psi_G:\mbox{Pt}_G(\mathbb C)\rightarrow \mbox{Act}_G(\mathbb C)$ are such that $\Xi_G.\Psi_G=\mathcal J_G$, the comparison functor of the adjunction. .
\end{prop}

\begin{proof}The first assertion is an immediate consequence of Proposition \ref{coequalizer}.
Let us now prove the assertion about the comparison adjunction. First, let us consider a $\mathbb T_G$-algebra $\xi:T_G(A)\rightarrow A$ which is in $\mathbb X_G$, i.e. which is the extension of a (unique) action $\psi:(A|G)\rightarrow G$. We show that $\eta'_{\xi}$ is an isomorphism between $\xi$ and ${\mathcal J}_G(F'(\xi))$. Consider the following diagram :
$$\xymatrix{
(A|G)\ar[r]^{j_{A,G}}\ar[dr]_{\psi}&T_GA\ar[d]_{\xi}\ar[r]^{k_{A,G}}&A+G\ar[d]^{q_{\psi}}\\
&A\ar[r]_{l_{\psi}}\ar[ru]_{i_A}&A\rtimes_{\psi}G\ar@<-2pt>[r]_-{p_{\psi}}&G\ar@<-2pt>[l]_-{s_{\psi}}}$$
Then in view of all what preceeds, $(A\rtimes_{\psi}G,p_{\psi},s_{\psi})=F'(\xi)$ and since $l_{\psi}$ is the kernel of $p_{\psi}$ (because $\psi$ is an action), $G'(F'(\xi))$ is the unique arrow $h$ from $T_GA$ to $A$ such that $l_{\psi}h=q_{\psi}.k_{A,G}$, i.e. $\xi$.

Secondly, we show that $G'$ has values in $\mathbb X_G$, i.e. that for any object
$$\xymatrix{
X\ar@<-2pt>[r]_p&G\ar@<-2pt>[l]_s}$$
the algebra $\mathcal J_G(X,p,s)$ is (the extension to $T_G(A)$ of)  an action $(A,G,\psi)$. Moreover we show that it is  such that $A\rtimes_{\psi}G$ is isomorphic to $(X,G,p,s)$, the $G$-part of the isomorphism being the identity : this in fact shows that $\epsilon'_{(X,p,s)}$ is an isomorphism.

Let $A$ be $\mbox{Ker}(p)$ and $l:A\rightarrow X$ be $\mbox{ker}(p)$. Consider the following diagram, where $p_G$ denotes the projection on $G$ in the product $A\times G$. It is commutative, since $p<l,s>=r_G$ (one may check it by preceeding these morphisms by the canonical injections $i_A$ and $i_G$):

$$\xymatrix{
(A|G)\ar[r]^{\iota_{A,G}}\ar@{.>}[dd]_{\exists !\psi}&A+ G\ar[dd]_{<l,s>}\ar@<2pt>[rdd]^{r_G}\ar[r]^{b_{A,G}}&A\times G\ar[dd]^{p_G}\\
\\
A\ar[r]_l\ar[uur]^{i_A}&X\ar[r]_p&G\ar@<2pt>[luu]^{i_G}
}$$
Since $\iota_{A,G}$ is the kernel of $b_{A,G}$ and since $p_G.b_{A,G}=p.<l,s>$, one has: $p.<l,s>.\iota_{A,G}=p_G.b_{A,G}.\iota_{A,G}=0$; so since $l$ is the kernel of $p$, there exists a unique $\psi:(A|G)\rightarrow A$ such that $l.\psi=<l,s>.\iota_{A,G}$. We claim that $(A,G,\psi)$ is an action. Since $l=<l,s>.i_A$ is a monomorphism, it suffices to show that $<l,s>$, which is known to be a regular epimorphism, is the coequalizer of $\iota_{A,G}$ and $i_A.\psi$.

Consider $q_{\psi}:A+ G\rightarrow Q_{\psi}$ the coequalizer of $\iota_{A,G}$ and $i_A.\psi$, and $p_{\psi}$ and $s_{\psi}$ defined as in Proposition \ref{coeq} above. We will show that $q_{\psi}.i_A$ is a monomorphism, thus showing that is an action, and that the $G$-point $(Q_{\psi},p_{\psi},s_{\psi})$, which is nothing but $A\rtimes_{\psi}G$, is isomorphic to $(X,p,s)$.

First of all, $<l,s>.i_{A}.\psi=l.\psi=<l,s>.\iota_{A,G}$ hence, since $q_{\psi}$ is the coequalizer of $i_A.\psi$ and $\iota$, there exists a unique morphism $e:Q_{\psi}\rightarrow X$ such that $q_{\psi}.e=<l,s>$:

$$\xymatrix{
&A+ G\ar[d]^{<l,s>}\ar@/^3pc/[dd]^{q_{\psi}}\\
A\ar[ur]^{i_A}\ar[r]^l\ar[dr]_{q_{\psi}.i_A}&X\\
&Q\ar[u]_e}$$
and of course, then $e.q_{\psi}.i_A=<l,s>.i_A=l$, so this diagram commutes. But then, since $l$ is the kernel of $p$, it is a monomorphism, hence so is $q_{\psi}.i_A$.
But then, by Proposition \ref{coeq}.3) above, $q_{\psi}.i_A$ = ker $p_{\psi}$.

Then consider the following diagram:
$$\xymatrix{
&&&A+ G\ar@/_1pc/[ddl]_{<l,s>}\ar[d]^{q_{\psi}}\ar@<-2pt>[r]_{r_G}&G\ar@<-2pt>[l]_{i_G}\ar@{=}[d]\ar@/_1pc/@{=}[ddl]\\
&1\ar[r]&A\ar@{=}[ld]\ar[r]^{q_{\psi}.i_A}&Q_{\psi}\ar[ld]^e\ar@<-2pt>[r]_{p_{\psi}}&G\ar@<-2pt>[l]_{s_{\psi}}\ar@{=}[ld]\ar[r]&1\\
1\ar[r]&A\ar[r]_l&X\ar@<-2pt>[r]_p&G\ar@<-2pt>[l]_s\ar[r]&1
}$$

The only part of this diagram which has not been shown to commute is the bottom right-hand square. But one has $ p_{\psi}.q_{\psi}=r_G=p.<l,s>=p.e.q_{\psi}$, hence $p.e=p_{\psi}$ since $q_{\psi}$ is a (regular) epimorphism. And $e.s_{\psi}=e.q_{\psi}.i_G=<l,s>.i_G=s$. Hence all conditions of the Short Split Five Lemma are satisfied, so $e$ is an isomorphism, and thus defines an isomorphism $(e,1_G)$ in $\mbox{Pt}(\mathcal C).$

Note that we then have implicitly constructed (by composition with the inverse of $\Xi_G)$ a functor $\mbox{Pt}(\mathcal C)\rightarrow \mbox{Act}(\mathcal C)$. We denote it by $\Psi$ (and its restriction to the fibers on $G$ by $\Psi_G$). For a point $\xymatrix{X\ar[r]_p&G\ar@/_/[l]_s}$, $\Psi(X, G,p,s)$ is an action on $A=\mbox{Ker }p$. It is easy to verify that if $\xymatrix{X'\ar[r]_{p'}&G'\ar@/_/[l]_{s'}}$ is another point and $(x,g)$ is a morphism of points between them, i.e. a pair of morphisms making the following diagram commute
$$\xymatrix{X\ar[r]_p\ar[d]_x&G\ar@/_/[l]_s\ar[d]^g\\
X'\ar[r]_{p'}&G'\ar@/_/[l]_{s'}}$$then $\Psi(x,g)$ is the unique map $a:A\rightarrow A'$ making the following diagram commute
$$\xymatrix{A=\mbox{Ker }p\ar[d]_a\ar[rr]^{\mbox{\tiny ker }p}&&X\ar[d]^x\\
A'=\mbox{Ker }p'\ar[rr]^{\mbox{\tiny ker }p'}&&X'}$$which indeed is a map of actions between $\Psi(X,G,p,s)$ and $\Psi(X',G',p',s')$.

Finally, the very constructions of $\Psi_G$ and $\Xi_G$ ensure that $\Xi_G\Psi_G=\mathcal J_G$.
\end{proof}

\noindent{\bf Examples} 

\noindent 1) Recall that the conjugation action of an object $E$ of $\mathbb C$ on itself is defined in \cite{BJ}, as a split extension on $E$, to be the short exact sequence  $0\to E \mr{\sigma_1} E \times E \mr{\pi_2} E\to 0$ with the splitting $\Delta: E \to E \times E$ being the diagonal map. By Proposition \ref{equivact}, it corresponds to some action $(E|E)\rightarrow E$. We show that this corresponding action, which  will be denoted by $c^E_2$ in the sequel, is $\nabla_E.\iota_{E,E}$, where $\nabla_E: E+E\rightarrow E$ is the codiagonal, and that the corresponding algebra $\Xi_E(c_2^E)$ is $\nabla_E.k_{E,E}$.  Since $\psi=\Xi_E(\psi).k_{E,E}$, it suffices to prove the second assertion, and  by  the construction of $\Xi_E$ it suffices to show that this map makes  the following diagram commute:
$$\xymatrix{
T_EE\ar[r]^{k_{E,E}}\ar[d]_{k_{E,E}}&E+E\ar[dd]^{<\sigma_1,\Delta_E>}\\
E+E\ar[d]_{\nabla_E}\\
E\ar[r]^{\sigma_1}&E\times E}$$%
Followed by $\pi_1$, one gets $\pi_1.\sigma_1.\nabla_E.k_{E,E}=\nabla_E.k_{E,E}$ on the one hand, and $\pi_1.<\sigma_1,\Delta_E>.k_{E,E}$ $=<\pi_1.\sigma_1,\pi_1.\Delta_E>.k_{E,E}=<1_E,1_E>.k_{E,E}=\nabla_E.k_{E,E}$ on the other hand. And followed by $\pi_2$, one gets $\pi_2.\sigma_1.\nabla_E.k_{E,E}=0.\nabla_E.k_{E,E}=0$ on the one hand, and $\pi_2.<\sigma_1,\Delta_E>.k_{E,E}$ $=<\pi_2.\sigma_1,\pi_1.\Delta_E>.k_{E,E}=<0,1_E>.k_{E,E}=r_2.k_{E,E}=0$ on the other hand.

Of course, this construction $c^{(-)}_2$ is functorial. More precisely,  $c^{(-)}_2$ may be considered as a functor $\mathbb C\rightarrow \mbox{Act}(\mathbb C)$, or as a natural transform $S\rightarrow 1_{\mathbb C}$, where $S:\mathbb C\rightarrow \mathbb C$ is defined by $F(E)=(E|E)$ and $S(f)=(f|f)$, both meaning that the following diagram commutes:
$$\xymatrix{
(E|E)\ar[d]_{c_2^E}\ar[r]^{(f|f)}&(F|F)\ar[d]^{c_2^F}\\
E\ar[r]_f&F}$$
\noindent which follows immediately from naturality of $\iota$ and of $\nabla$.

 Then the split extension
$$\xymatrix{0\ar[r] &E\ar[r]^-{i_1} &E\rtimes_{c_2^E} E \ar[r]^-{p_2}   &E\ar[r]&0}$$
is (canonically isomorphic to) the following one:
$$\xymatrix{0\ar[r] &E \ar[r]^{i_1} & E \times E \ar[r]^{p_2} & E\ar[r] &0
}$$
\noindent by the very definition of $c_2$. We will generalize this construction to a conjugation action of an object on any  proper subobject in the following paragraph.\\

\noindent 2) The split short exact sequence of Lemma \ref{inclcrosseff}:
$$\xymatrix{
0\ar[r]&(A|G)\ar[r]^j&T_G(A)\ar[r]_{r_A.k}&A\ar@/_1pc/[l]_{\eta_A}\ar[r]&0
}$$
corresponds to an action $\psi$ of $A$ on $(A|G)$ such that $T_G(A)=(A|G)\rtimes_{\psi}A$, and the following one:
$$\xymatrix{0\ar[r]&T_G(A)\ar[r]^k&A+G\ar[r]^{r_G}&G\ar@/_1pc/[l]_{i_G}\ar[r]&0}$$ to an action $\psi'$ of G on $T_G(A)$ such that $A+G=T_G\rtimes_{\psi'}G$,
hence one may write $A+G=((A|G)\rtimes_{\psi}A)\rtimes_{\psi'}G$.

\begin{prop}\label{stabact} Let $\psi:(A|G) \to A$ be an  action in $\C$, and let $B\mr{b}A$ and $h:H\to G$ be two subobjects. Suppose that $B$ is $H$-stable under $\psi$, i.e.\ the map $\psi(b|h): (B|H) \to A$ factors through a map $\psi': (B|H) \to B$ such that $b\psi'=\psi(b|h)$. Then $\psi'$ is an action of $H$ on $B$.\end{prop}

\begin{proof}Consider the following diagram of solid arrows, where $q'$ is the coequalizer of $\iota_{B,H}$ and $i_B\psi'$; all squares and the bottom triangles are commutative, and the upper triangles are coequalized by $q'$ and $q=q_{\psi}$ respectively:

$$\xymatrix{
&B+H\ar[dd]^{q'}\ar[rrr]^{b+h}&&&A+G\ar[dd]^q\\
(B|H)\ar[dd]_{\psi'}\ar[rrr]^{(b|h)}\ar[ru]^{\iota_{B,H}}&&&(A|G)\ar[dd]^{\psi}\ar[ru]^{\iota_{A,G}}\\
&Q \ar@{.>}[rrr]^f&&&A\rtimes_{\psi}G\\
B\ar[ruuu]_{i_B}\ar[rrr]_b\ar[ru]_{q'.i_B}&&&A\ar[ruuu]_{i_A}\ar[ru]_{l_{\psi}}
}$$

\noindent We have to show that $q'i_B$ is a monomorphism. Since $q(b+h)i_B\psi'=qi_Ab\psi'=qi_A\psi(b|h)=q\iota_{A,G}(b|h)=q(b+h)\iota_{B,H}$, there is a unique $f:Q\rightarrow A\rtimes_{\psi}G$ such that $q(b+h)=fq'$. Then $l_{\psi}b=qi_Ab=q(b+h)i_B=fq'i_B$. Hence, since $b$ and $l_{\psi}$ are monomorphisms, so is $q'i_B$.
\end{proof}
\noindent Note that under these conditions, $Q'$ is $B\rtimes_{\psi'}H$ and $f$ is nothing but $b\rtimes h$; and it is easy to see that it is a monomorphism too.

Finally, we need to exhibit the kernel and image of the semi-direct product of compatible maps. Here, and several times in the sequel, we need the following obvious lemma.

\begin{lemma}\label{factorlem}
Consider a commutative diagram
\[\xymatrix{
X\ar[r]^-f \ar@{->>}[d]^q & Y\ar@{{ >}->}[d]^i\\
Q \ar[r]^j & Z
}\]
in any category where $q$ is a regular epimorphism and $i$ is a monomorphism. Then $j$ factors through $i$.
\end{lemma}

\begin{prop}\label{KerImfxg} Let $\psi:(A|G) \to A$ and $\psi': (B|H)\to B$ be  actions in $\C$, and let $f:A\to B$ and $g:G\to H$ in $\C$ such that $(f,g)$ is a map $\psi\rightarrow \psi'$ in $\mbox{Act}(\mathbb C)$, i.e.\ the diagram
\begin{equation}\label{fxgdia}
\xymatrix{ (A|G)\ar[r]^{\psi} \ar[d]^{(f|g)} & A\ar[d]^f\\
(B|H)\ar[r]^{\psi'} & B}
\end{equation}
commutes. Then 
\[\mbox{ker} (f\rtimes g) =   \mbox{ker}(f) \rtimes\mbox{ker}(g): \mbox{Ker}(f) \rtimes_{\tilde{\psi}} \mbox{Ker}(g) \to A\rtimes_{\psi} G\]
where $\tilde{\psi}$ is given by restricting to Ker$(g)$ the $G$-action on Ker$(f)$ coming from the  fact that Ker$(f)$ is stable under $\psi$, see Proposition \ref{stabact}. Analogously,
%\[\mbox{Im} (f\rtimes g) =   \mbox{Im}(f) \rtimes_{\tilde n} \mbox{Im}(g)\]
\[ \mbox{im} (f\rtimes g) =   \mbox{im}(f) \rtimes \mbox{im}(g): \mbox{Im}(f) \rtimes_{\tilde{\psi}' }\mbox{Im}(g) \to B\rtimes_{\psi'} H\]
Here $\tilde{\psi}':( \mbox{Im}(f)| \mbox{Im}(g))\to  \mbox{Im}(f)$ is the unique map such that  
%$\tilde n (\tilde f|\tilde g) = \tilde f m$ 
$\mbox{im}(f) \tilde{\psi}'  =\psi' (\mbox{im}(f)|\mbox{im}(g))$
where $\xymatrix{A\ar@{->>}[r]^-{\tilde f} & \mbox{Im}(f) \ar@{{ >}->}[r]^-{\mbox{\tiny im }f} &B}$ denotes an image factorization of $f$ and similarly for $g$.
\end{prop}

\begin{proof} First note that diagram (\ref{fxgdia}) assures the existence of $\tilde{\psi}$ and $\tilde{\psi}'$: stability  of Ker$(f)$ under $\psi$ is easily deduced; to see that $\psi' (\mbox{im}(f)|\mbox{im}(g))$ indeed factors through $\mbox{im}(f)$, apply Lemma \ref{factorlem} to the factorization $(f|g)=(\mbox{im}(f)|\mbox{im}(g))(\tilde{f}|\tilde{g})$, noting that $(\tilde{f}|\tilde{g})$ is a regular epimorphism by Propositions \ref{cr-reg-epi} and \ref{elcrprops}.(3).

Now it is clear that 
\[(\mbox{ker}(f),\mbox{ker}(g)): (\mbox{Ker}(f),\mbox{Ker}(g),\tilde{\psi}) \longrightarrow (A,G,\psi)
\]
 is a kernel of the map $(f,g)$ in $\mbox{Act}({\mathbb C})$, whence 
 \[(\mbox{ker}(g),\mbox{ker}(f) \rtimes \mbox{ker}(g)):  (\mbox{Ker}(g), \mbox{Ker}(f)\rtimes_{\psi'} \mbox{Ker}(g), p_{\psi'},s_{\psi'})
 \to (G,A\rtimes_{\psi}G,p_{\psi},s_{\psi})\]
  is a kernel of $f\rtimes g$ in $\mbox{Pt}({\mathbb C})$, by the equivalence 
$\mbox{Act}({\mathbb C}) \sim \mbox{Pt}({\mathbb C})$. To show that $\mbox{ker}(f) \rtimes \mbox{ker}(g)$ is a kernel of $f\rtimes g$ in $\mathbb C$, let $x:X\to A\rtimes_{\psi}G$ be a map in $\mathbb C$ such that $(f\rtimes g)x=0$.
Then $p_{\psi}x$ factors as $X\xrightarrow{\tilde x} 
\mbox{Ker}(g) \xrightarrow{\rm{ker}(f)} G$ for some map $\tilde x$, and
\[\xymatrix{
X+\mbox{Ker}(g) \ar[d]_{<x,s_{\psi}\rm{ker}(g)>} \ar@<-2pt>[r]_-{<\tilde{x},1>} &\mbox{Ker}(g)\ar@<-2pt>[l]_-{i_2}\ar[d]^{\rm{ker}(g)}\\
A\rtimes_{\psi}G \ar@<-2pt>[r]_-{p_{\psi}}&G\ar@<-2pt>[l]_-{s_{\psi}}
}\]
is a map in $\mbox{Pt}(\mathbb C)$ whose postcomposition with $(g,f\rtimes g)$ is trivial, hence factors through $(\mbox{ker}(g),\mbox{ker}(f) \rtimes \mbox{ker}(g))$. Consequently $x$ factors through $\mbox{ker}(f) \rtimes \mbox{ker}(g)$, as desired.

To see that $\mbox{im}(f) \rtimes \mbox{im}(g)$ is an image of $f\rtimes g$, it suffices to note that $f\rtimes g = (\mbox{im}(f) \rtimes \mbox{im}(g))(\tilde f \rtimes \tilde g)$ where $\tilde f \rtimes \tilde g$ is a regular epimorphism and 
$\mbox{im}(f) \rtimes \mbox{im}(g)$ is a monomorphism by the short five lemma for regular epimorphisms and for monomorphisms, resp.
%$\mbox{im}(f) \rtimes \mbox{im}(g)$ is monic since it has a trivial kernel by the nine-lemma.
\end{proof}

\medskip

\section{Conjugation action of an object on a proper subobject}

 Using our alternative description of internal actions we now introduce a  general notion of conjugation action of an object $E$   on any of its  proper subobjects, as follows.

\begin{prop}\label{conjnormsub} Let $n: N \to E$ be a proper subobject in $\mathbb C$. Then there is an action $c^{N,E}: (N|E) \to N$ of $E$ on $N$ such that $nc^{N,E} = c_2^E(n|1)$. We call $c^{N,E}$   the \textit{conjugation action} of $E$ on $N$. It is natural with respect to pair maps $(E,N) \to (E',N')$, i.e.\ maps $f:E\to E'$ in \C\ such that $f(N)\subset N'$ for a given subobject $N'$ of $E'$.
\end{prop}

\begin{proof} Let $\pi:E\to G$ be a cokernel of $n$.
Then we have  the following  diagram of plain arrows which commutes by naturality of commutator maps:
$$\xymatrix{
(N|E)\ar@{.>}[d]^{c^{N,E}} \ar[r]^{(n|1)} & (E|E)\ar[r]^{(\pi|\pi)} \ar[d]^{c_2^E} & (G|G)\ar[d]^{c_2^G}\\
N \ar[r]^n & E\ar[r]^{\pi} & G}$$
Thus $\pi c_2^E (n|1) =c_2^G(\pi|\pi)(n|1) = c_2^G(\pi n|\pi)  = (0|\pi)=
0$ since the functor $(-|-)$ is bireduced by Proposition \ref{elcrprops} (2), whence $c_2^E (n|1)$ factors through $n$, thus providing the desired map $c^{N,E}$. By Proposition \ref{stabact} it is an action.
%In order to prove that it is an action consider the map $<n,1>: N+ E \to E$. We check that it coequalizes the injection $\iota: (N|E) \to N+ E$ and the map $i_N c^{N,E}$. Indeed, $<n,1>\iota = <1,1>(n+ 1)\iota=
%<1,1>\iota (n|1) =c_2^E  (n|1) = n c^{N,E} =
%<n,1>i_N c^{N,E}$. Thus $<n,1>$ factors through a map $\overline{<n,1>}:Q\to E$ on the coequalizer $Q$ of $\iota$ and $i_N c^{N,E}$. But $\overline{<n,1>}i_N^{\prime}=n$ which is monic, hence so is $i_N^{\prime}$. Thus $c^{N,E}$ is an action, as desired.
 Its naturality is immediate from naturality of $c_2^E$.
\end{proof}

Note that $c^{E,E}$ coincides with the conjugation action of $E$ on itself in Example 1 following Proposition \ref{compar}, so our definition of a conjugation action of $E$ on $N$ generalizes the one for $N=E$ in \cite{BJ}.

%\begin{prop} Suppose that $\mathbb C$ is exact, and let $a:A\hra X$ be a subobject in $X$. Then
%the {\em normal closure} Ker(Coker$(a)$) of $A$ in $X$,  denoted by $\lhd A \rhd_X$, is...
We now give two properties which litterally generalize certain standard facts in the theory of groups or Lie algebras.

First we quote a reformulation of the \textit{universal property} of the semi-direct product:

%Next, forming the semi-direct product allows to realize an abstract action  as a  conjugation action. More precisely, we have:
\begin{prop}\label{PU} Let $\psi:(A|G) \to A$ be an  action in $\C$, and let $A \mr{f} X \ml{g} G$ be maps in \C. Then there exists a map $h=\overline{<f,g>}: A \rtimes_{\psi} G \to X$ such that $hl_{\psi} =f$ and $hs_{\psi} =g$ iff the following square commutes:
\begin{equation}\label{PUdia}
\xymatrix{ (A|G)\ar[r]^{\psi} \ar[d]_{(f|g)} & A \ar[d]^f\\
(X|X) \ar[r]^{c_2^X} & X}
\end{equation}
Moreover, if $h$ exists it is unique.
\end{prop}

\begin{proof} Uniqueness follows from the fact that $(l_{\psi},s_{\psi})$ is a (strongly) epimorphic pair. To prove existence, by the definition of $A \rtimes_{\psi} G$ as a quotient of $A+ G$, we must show that commutativity of diagram \REF{PUdia} is equivalent with the relation $<f,g>\iota = <f,g>i_A{\psi}$. But $<f,g>\iota = \nabla (f+ g)\iota = \nabla\iota (f|g) =c_2^X (f|g)$. On the other hand, $<f,g>i_A{\psi} = f{\psi}$, whence the assertion.
\end{proof}

In particular, Proposition \ref{PU} shows that the semi-direct product  can be viewed as a universal  transformation of an abstract action into a conjugation action:

\begin{cor}\label{m=conj}  An  action ${\psi}:(A|G) \to A$   in $\C$ coincides with the restriction to $G$ of the conjugation action of $A\rtimes_{\psi} G$ on $A$, or formally, $c^{A,A\rtimes_{\psi} G} (1|s_{\psi}) = {\psi}$.
\end{cor}

\begin{proof} In Proposition \ref{PU}, take $X=A\rtimes_{\psi}G$, $f=l_{\psi}$, $g=s_{\psi}$ and $h=1$. We get $l_{\psi}{\psi} = c_2^X(l_{\psi}|s_{\psi}) = c_2^X(l_{\psi}|1) (1|s_{\psi}) = l_{\psi}c^{A,X}(1|s_{\psi})$, whence the assertion since $l_{\psi}$ is monic.
\end{proof}

\medskip

To make further progress we define commutators of subobjects in terms of the cross-effects of the identity functor; this actually is the starting point of 
a new approach to categorical commutator calculus  which is further developped in \cite{CCC}. 

For $n\ge1$ and an object $X$ in $\mathbb C$, let $\nabla^n_X =<1,\ldots,1> :nX=X+\ldots+X \to X$ be the canonical folding map.

%Now let $\mathbb C$ be a finitely cocomplete homological category. 

\begin{defi}\label{comdef} The $n$-fold commutator map of an object $X$ of $\mathbb C$ is the natural composite map
\[ \xymatrix{c_n^X : (X|\ldots|X) \ar@{{ >}->}[r]^{\iota} & X+\ldots+X \ar[r]^-{\nabla^n_X} & X}\]
Moreover, if $x_i: X_i \hookrightarrow X$ are subobjects of $X$, define their commutator to be the following subobject of $X$:
\[ [X_1,\ldots,X_n] = {\rm Im}((X_1|\ldots|X_n)\mr{(x_1|\ldots|x_n)}(X|\ldots|X) \mr{c_n^X} X).\]
\end{defi}

\begin{rem} With the notations of definition \ref{comdef}, the Huq commutator of $X_1,X_2$ is the proper closure of our commutator $[X_1,X_2]$ defined above (where the proper closure of a subobject is defined to be the kernel of its cokernel). Thus, if $X_1,X_2$ generate $X$ (i.e. if the map $< x_1,x_2>: X_1  + X_2 \to X$ is a regular epimorphism), then $[X_1,X_2]$ coincides with the commutators of $X_1,X_2$ of Huq and of Smith (the Smith case being due to Everaert and van der Linden (private communication)). In an unpublished note \cite{B}, Dominique Bourn gives a detailed proof of this fact in the special case when $X_1$ and $X_2$ are considered as subobjects of their sum $X_1 + X_2$.

\end{rem}

%Note that the conjugation action of $E$ on itself coincides with its $2$-fold commutator map ($c^{E,E}=c_2^E$). 

The link with the subject of this paper is that stability under the conjugation action can be expressed in terms of commutators:

\begin{lemma}\label{normalizeequ} Let $\xymatrix{
X\ar[r]^-{x} &A &Y\ar[l]_-y }$ be subobjects of an object of  $\mathbb C$. Then $X$ is $Y$-stable under the  conjugation  action of $A$ on itself (see  Proposition \ref{stabact}) iff $[X,Y]\subset X$, i.e. the injection of $[X,Y]$ into $A$ factors through $x$.
\end{lemma}

\begin{proof} Consider the following commutative diagram of solid arrows:

$$\xymatrix{
(X|Y)\ar@{.>}[dd]_{\psi}\ar[rr]^{(x|y)}\ar@{->>}[rd]^q&&(A|A)\ar[dd]^{c_2^A}\\
&[X,Y]\ar@{.>}[ld]_u\ar@{>->}[rd]^i\\
X\ar@{{ >}->}[rr]^x&&A
}$$
Then by the respective definitions, $X$ is $Y$-stable iff a map $\psi$ as indicated exists and renders the diagram commutative; and $[X,Y]\subset X$ iff a map $u$ as indicated exists and renders the diagram commutative. But these conditions are equivalent by Lemma \ref{factorlem} since $q$ is a regular epimorphism and $x$ is monic.
\end{proof}
%Then the projection $q$ is the cokernel of its kernel, which is also the kernel of $c_2^A. (x|y)$. So if there is a $\psi$ making the square commute, then the kernel of $\psi$ also is the kernel of $q$, and since $q$ is the cokernel of this kernel, there exists a unique $u$ such that $uq=\psi$. Then $xuq=x\psi=iq$, and since $q$ is a (regular) epi, $xu=i$. Conversely, if there exists $u$ such that $xu=i$, then $xuq=iq=c_2^A.(x|y)$, so $\psi=uq$ makes the square commute.

\begin{defi}\label{normalize} Let $\xymatrix{
X\ar[r]^-{x} &A &Y\ar[l]_-y }$ be subobjects of an object of  $\mathbb C$. We say that $Y$ normalizes $X$ if $X$ is $Y$-stable under the  conjugation  action of $A$ on itself (i.e. if $[X,Y]\subset X$, in view of Lemma \ref{normalizeequ}).
\end{defi}

The following is an extremely useful criterion of normality in semi-abelian categories, as will be shown in the sequel and in subsequent work.

\begin{theorem}\label{propercrit} Suppose that $\mathbb C$ is semi-abelian. Then the following properties are equivalent for subobjects $X,Y$ of $A$ as above:

\begin{enumerate}

\item $Y$ normalizes $X$.

\item $X$ is a proper subobject of $X\vee Y$, the subobject generated by $X$ and $Y$ (i.e.\ the image of the map $<x,y>\colon X+Y \to A$).

\item The object $X\cap Y$ is proper in $Y$ and the sequence
\[ \xymatrix{0 \to X \ar[r]^{x'} &X\vee Y \ar[r]^-{qr_Y} & Y/X\cap Y \to 0}\]
is  short exact  
where $x'$ is the factorization of $x$ through $X\vee Y$ and $q\colon Y\to Y/X\cap Y$ is the projection.
\end{enumerate}

If one of these conditions (1)-(3) is satisfied we write $X\vee Y =[X]Y =Y[X]$.

\end{theorem}

We note that the implication (1) $\Rightarrow$ (3) is a crucial ingredient in our commutator theory in \cite{CCC}.\medskip

%The proof requires the following little lemma.

%\begin{lemma}\label{stabcjg}Let $x\colon X \hookrightarrow A$ be a subobject of $A$. Then $X$ is $A$-stable under the conjugation action of $A$ (see Proposition \ref{stabact}) if and only if $[X,A]$ is a subobject of $X$ (more precisely as a subobject of $A$, i.e. the inclusion of $[X,A]$ in $A$ factorizes through $x$).\end{lemma}

\noindent{\bf Proof:}
The equivalence of (2) and (3) is immediate using the second Noether isomorphism theorem. Moreover, (2) implies (1) by Proposition \ref{conjnormsub}, so it remains to show that (1) implies (2). By Proposition \ref{stabact} we obtain an action $\psi$ of $Y$ on $X$, and the universal property (Proposition \ref{PU}) of the semi-direct product implies that there is a map $\overline{<x,y>}\,\colon\, X\rtimes_{\psi} Y \to A$. Now ${\rm Im}(\overline{<x,y>}) = {\rm Im}({<x,y>}) = X\vee Y$ since by construction $X\rtimes_{\psi} Y$ is a quotient of $X+Y$. But ${\rm Im}(l_{\psi})$ is proper in 
$X\rtimes_{\psi} Y$, whence ${\rm Im}(\overline{<x,y>} l_{\psi})={\rm Im}(x) =X$ is proper in 
${\rm Im}(\overline{<x,y>}) = X\vee Y$ by exactness of $\mathbb C$.\hfill$\Box$\medskip

\begin{cor}\label{critprop} A finitely cocomplete homological category $\mathbb C$ is semi-abelian (i.e.\ Barr exact) iff the following condition holds:\medskip

\noindent(P) A subobject $X$ of an object $A$ of $\mathbb{C}$ is proper in $A$ iff it is stable under the conjugation action of $A$, i.e.\ if $[X,A]\subset X$. 

\end{cor}

The result that property (P) holds in semi-abelian categories was recently also obtained independantly by Mantovani and Metere \cite{MM}.

%\end{document}
\begin{proof} The condition is necessary by Theorem \ref{propercrit} (take $Y=A$); let us prove that it is sufficient: to show that $\mathbb{C}$ is exact, let 
$\xymatrix{ X\ar@{{ |>}->}[r]^-{x} &A\ar@{-{>>}}[r]^-q &B }$ be a proper subobject and a regular epimorphism in 
$\mathbb{C}$. Let $\xymatrix{X \ar@{>>}[r]^-{q_X} & qX \ar@{{ >}->}[r]^-{\bar{x}} & B}$ be an image factorization of $qx$. Then we get a commutative 
diagram
\[ \xymatrix{
(X|A) \ar[dd]^{(q|q_X)} \ar[rr]^{c^{A,X}} \ar[rd]^{(1|x)}
& & X\ar[dd]^{q_X} \ar[rd]^{x} \\
 & (A|A) \ar[dd]^{(q|q)} \ar[rr]^{c_2^A} & & A\ar[dd]^q \\
 (B|qX) \ar[dr]^{(1|\bar{x})} & & qX \ar[dr]^{\bar{x}} & \\
  & (B|B) \ar[rr]^{c_2^B} & & B
}\]
It implies that $c_2^B(1|\bar{x})(q|q_X) = \bar{x}q_Xc^{A,X}$, whence $c_2^B(1|\bar{x})(q|q_X)$ factors through $\bar{x}$. But $(q|q_X)$ is a regular epimorphism by Propositions \ref{cr-reg-epi} and \ref{elcrprops}.(3), whence $c_2^B(1|\bar{x})$ also factors through $\bar{x}$, compare the proof of Lemma \ref{normalize}. Consequently $qX$ is stable under the conjugation action of $B$, thus it is proper by hypothesis.
\end{proof}

Note that since in a semi-abelian category ``proper subobject'' is equivalent to ``normal subobject'', the equivalence of (1) and (2) in Theorem \ref{propercrit} may obviously be stated as the equivalence of (1) and the following \\

\noindent(2'): {\em  $X$ is a normal subobject of $X\vee Y$.}\\

The proof of (1) $\Rightarrow$ (2) in a semi-abelian category indeed provides a proof of (1) $\Rightarrow$ (2') in any finitely cocomplete homological category, since in any such category the image of any proper subobject (or even normal subobject) by a regular epimorphism is a normal subobject. We do not know if the converse (2') $\Rightarrow$ (1) is true in any such category, but the following example shows that the latter implication may be true even if the category is not exact (hence the existence of the conjugation action of an object $X$ on a normal subobject $N$ does not imply in general that this subobject is proper, by corollary \ref{critprop}).

First note that (2') $\Rightarrow$ (1) holds in a finitely cocomplete homogical category if and only if it holds in the special case when $Y=A$ (and $y=1_A$), i.e. if and only if any object acts by conjugation on any of its normal subobjects. Indeed, suppose that this is true, and let $x:X\rightarrow A$ and $y:Y\rightarrow A$ be subobjects such that $X$ is normal in $X\vee Y$. Then $X\vee Y$ acts on $X$ by conjugation, by hypothesis. So we get an action $\psi : (X|X\vee Y)\rightarrow X$ which is the restriction on $X$ of the conjugation action of $X\vee Y$ on itself. Then by precomposition with $(1|y')$ (where $y'$ denotes the inclusion of $Y$ in $X\vee Y$) we get a map $(X|Y)\rightarrow X$ which is the restriction to $(X|Y)$ of the conjugation action of $X\vee Y$ on itself, but also of $A$ on itself.

Consider the category $\mathbb C$ whose objects are ordered pairs $\mathcal G=(G,B)$ where $G$ is a group and $B$ a subgroup. A map $f:\mathcal G=(G,B)\rightarrow\mathcal H=(H,C)$ is a group homomorphism $f:G\rightarrow H$ such that $f(B)\subseteq C$. We show that in this category, any object $\mathcal G$ acts by conjugation on any of its normal subobjects. This category is (finitely) cocomplete homological and is provided with two ``forgetful'' functors $U_1$ and $U_2$ to the category of groups: $U_1(G,B)=G$, $U_1(f)=f$, $U_2(G,B)=B$, $U_2(f)=f|_B$. Both $U_1$ and $U_2$ preserve finite limits and coproducts. So for instance, $(\mathcal G|\mathcal H)$= $((G|H),(B|C))$ (where in the second member of the equality the $(.|.)$'s are computed in the category of groups). However, only $U_1$ preserves coequalizers: for instance, the cokernel of the inclusion of $(N,C)$ in $(G,B)$, where $N$ is, say, a normal subgroup of $G$ and $C$ a subgroup of $B$ (possibly, but not necessarily normal in $B$) is independent of $C$, it is $(G/N, B/(N\cap B))$. In particular, $(N,C)$ is not the kernel of its cokernel unless $C=N\cup B$, whence $\mathbb C$ is not semi-abelian. Now, an equivalence relation on $(G,B)$ in $\mathbb C$ is (the inclusion into $(G\times G,B\times B)$ of) a pair $(R,S)$, where $R$ is a congruence on $G$, $S$ is a congruence on $B$, and $S\subseteq R$. Like in any pointed protomodular category, a normal subobject of $(G,B)$ is ``the equivalence class of 0 for such an equivalence relation'', i.e. the inverse image of $R$ by the inclusion $\sigma_1$ (or equivalently $\sigma_2$) of ($G,B)$ in $(G,B)\times (G,B)$, i.e. the pair $\mathcal N=([e]_R,[e]_S)$ where $e$ is the unit of $G$ and $[e]_R$ its equivalence class (in $G$) for $R$, and $[e]_S$ in $B$ for $S$). The object $(\mathcal G|\mathcal G)$ is the pair $((G|G),(B|B))$; the elements of $(G|G)$ are products (in $G+G$) of the form $\prod_i[g_i,g'_i]^{z_i}_+$ with $g_i$ and $g'_i$ are in the two different copies of $G$, and where $[-,-]_+$ denotes commutators in $G+G$; and $(B|B)$ is the same with the $g_i$'s and $g'_i$'s in (the two copies of) $B$. The diagram that has to be filled up to get the result is the following:
$$\xymatrix{(\mathcal N|\mathcal G)=(([e]_R|G),([e]_S|A))\ar@{.>}[d]\ar[rr]&&(\mathcal G|\mathcal G)=((G|G),(A|A))\ar[d]^{c_2^{\mathcal G}}\\
\mathcal N=([e]_R,[e]_S)\ar[rr]&&\mathcal G=(G,A)}$$%
where the horizontal arrows are the  obvious inclusions. %
Moreover, $c_2^{\mathcal G}(\prod_i[g_i,g'_i]^{z_i}_+)=\prod_i[g_i,g'_i]^{z_i}$, where $[-,-]$ denotes the ordinary commutator in $G$. %
So what has to be proved is that if the  $g'_i$'s are in $[e]_R$ then $\prod_i[g_i,g'_i]^{z_i}$ is in $[e]_R$, and that if moreover the $g'_i$'s are in $[e]_S$ (hence in $A$) and the $g_i$'s are in $A$, then $\prod_i[g_i,g'_i]^{z_i}$ is in $[e]_S$. This is obvious, since $R$ and $S$ are congruences on $G$ and on $A$ respectively.

\section{Other characterizations of actions}

In this paragraph $\mathbb C$ remains a finitely cocomplete homological category, even if, at some places, we shall pay special attention to the  semi-abelian case. The functor $\Xi_G$ of Proposition \ref{compar} associates to any action $\psi$ (in our sense) of an object $G$ of $\mathbb C$ on another object $A$ an algebra $(A,\xi)$ over the monad $\mathbb T_G$, which moreover makes the following diagram commute.
$$\xymatrix{(A|G)\ar[rr]^j\ar[rd]_{\psi}&&T_G(A)\ar[ld]^{\xi}\\
&A}$$
Using our preceeding results, we here give  necessary and sufficient conditions for an arbitary  morphism $\psi:(A|G)\rightarrow A$ to have such an extension to an algebra over $\mathbb T_G$, so they are necessary for $\psi$ to be an action. Hence if the category $\mathbb C$ is such that the comparison functor $\mathcal J_G$ is an equivalence of categories between $\mbox{Pt}_G(\mathbb C)$ and $\mathbb C^{\mathbb T_G}$  (in particular if $\mathbb C$ is semi-abelian) they are also sufficient, but we also give a direct proof of this fact in a semi-abelian category, providing an alternative proof (without  using Beck's criterion) of the fact that $\mathcal J_G$ is an equivalence of categories in this case. %These conditions appear as three commutative diagrams.  In the case of the category of groups, if one associates to a morphism $\psi:(A|G)\rightarrow A$ (not supposed to be an action) the morphism $\phi:G\times A\rightarrow A$ given by $\phi(g,a)=\psi(g,a).a$ (for $a$ and $g$ different from the unit) , $\phi (1,a)=a$ and $\phi (g,1)=1$ as above , the first diagram means that for fixed $g$, $\phi(g,-)$ is an endomorphism of $A$, and the second that the map $G\rightarrow \mbox{End }A, g\mapsto \phi(g,-)$ is a morphism of monoids (hence takes values in $\mbox{Aut }A$ and, seen as a moprhism from $G$ to $\mbox{Aut }A$, is then a group morphism). So the first two axioms suffice in this case to characterize actions, hence the third one is automatically verified. The problem of characterizing a wider class of categories in which this remains true is open.

The object $G$ is fixed throughout this section; we therefore simplify the notations,   denoting by $\mathbb T$ the monad, by $T$ the underlying endofunctor of $\mathbb C$, and by $\eta$ and $\mu$ the unit and the multiplication of $\mathbb T$. We divide the question of the existence of an extension of a  map $\psi:(A|G)\rightarrow A$ to an algebra over $\mathbb T$ in two parts:

\begin{prop}\label{extunit}
Let $\psi$ be a map $(A|G)\rightarrow A$ in $\mathbb C$. Then the following conditions are equivalent:

\noindent 1. $\psi$ can be extended along $j$ to a map $\xi:T(A)\rightarrow A$ satisfying the unit axiom of a $\mathbb T$-algebra, i.e. $\eta_A.\xi=1_A$.

\noindent 2. The following diagram commutes:
\begin{equation}\label{extofpsidia}
\xymatrix{
((A|G)|A)\ar[r]^{\psi_0}\ar[d]_{(\psi|1_A)}&(A|G)\ar[d]^{\psi}\\
(A|A)\ar[r]_{c_2^A}&A
}\end{equation}
where $\psi_0$ is the action of $A$ on $(A|G)$ which arises form the fact that  $T(A)$ is the semi-direct product of $(A|G)$ and $A$ (see Lemma \ref{inclcrosseff} and Example 2 in section \ref{genprop}). Recall that in view of Corollary \ref{m=conj}, $\psi_0$ is the restriction to $A$ of the conjugation action of $T(A)$ on $(A|G)$, which itself is the restriction to $T(A)$ of the conjugation action of $A+G$ on $(A|G)$.

Moreover under these conditions,  the resulting map $\xi$ also satisfies $l_{\psi}.\xi=q_{\psi}.k$, where  $q_{\psi}:A+G\rightarrow Q_{\psi}$ is the coequalizer of $\iota_{A,G}$ and $i_A.\psi$ and $l_{\psi}$is the composite $q_{\psi}.i_A$ as in Proposition \ref{coeq}. \end{prop}

\begin{proof} That some $\xi$ extends $\psi$ along $j$ with $\xi.\eta_A=1_A$ can be translated into the following diagram:

$$\xymatrix{
(A|G)\ar[r]^j\ar[rd]_{\psi}&T(A)\ar[d]^{\xi}&A\ar[l]_{\eta_A}\ar@{=}[ld]\\
&A
}$$
so  by Proposition \ref{PU}, it is immediate that such an extension exists if and only   diagram (\ref{extofpsidia}) commutes. Moreover,  since the pair $(j,\eta_A)$ is a (strongly) epimorphic family, in order to prove that  $l_{\psi}.\xi=q_{\psi}.k$, it suffices to verify that $l_{\psi}.\xi.j=q_{\psi}.k.j$ and $l_{\psi}.\xi.\eta_A=q_{\psi}.k.\eta_A$. The first equality amounts to $q_{\psi}i_A.\psi=q_{\psi}.\iota$, which is true because $q_{\psi}$ is the coequalizer of these two maps, and the second to $q_{\psi}.i_A.\xi.\eta_A=q_{\psi}.k.\eta_A$, which is true because $\xi.\eta_A=1_A$ and $k.\eta_A=i_A$. 

\end{proof}

\noindent{\bf Example:} In the category of groups, consider a homomorphism $\psi:(A|G)\rightarrow A$. Since $(A|G)$ is the free group on $A^*\times G^*$, $\psi$ is the extension of a  unique set-application $\phi':A^*\times G^*\rightarrow A$. Let us define  $\phi:A\times G\rightarrow A$ by $\phi(e_G,a)=a$, $\phi(g,e_A)=e_A$ (where $e_G$ and $e_A$ denote the units\footnote{We prefer to denote the units this way, to avoid confusion with identities; but when dealing with maps, we denote by $0$ the zero maps, and by $1_X$ the identity on $X$}, and $\phi(g,a)=\phi'(g,a).a$ for non-trivial $g,a$ (so that of course $\psi([g,a])=\phi'(g,a)=\phi(g,a).a^{-1}$ for nonzero $g,a$). We know that $\psi$ is an action in our sense if and only if $\phi$ is an action in the usual sense. More generally, and rather starting our from $\phi$, one may determine the condition on $\phi$ which insures that $\psi$ satisfies the equivalent conditions of Proposition \ref{extunit}. In view of the proof of this proposition, it is more convenient to apply the classical universal property of the semi-direct product than condition 2. Considering that the action $\psi_0$ of Proposition \ref{extunit} corresponds to a classical action $\phi_0$, the property that is needed to insure the existence of $\xi$ is: for any $x\in(A|G),\psi(\phi_0(a,x))=a.\psi(x).a^{-1}$ (note that we indeed  work with $\psi$ on one hand, and with $\phi_0$ on the other hand). But since $\phi_0(a,-)$ and conjugation by $a$ both are group morphisms, it suffices to prove this for an $x$ of the form $[g,a']$. And $\phi_0(a,[g,a'])=a.[g,a'].a^{-1}$ (in $T(A)$, or in $A+G$). But, working in $A+G$, using the well-known formula $x.[y,z]=[x,y].[y,xz].x$ (or equivalently $[x,[y,z]]=[x,y][y,xz][z,y]$), one gets: % $a.[g,a'].a^{-1}=aga'g^{-1}a'^{-1}a^{-1}=[a,g]g(aa')g^{-1}(aa')^{-1}=[g,a]^{-1}g(aa')g^{-1}(aa')^{-1}=[g,a]^{-1}[g,aa']$. 
$a.[g,a'].a^{-1}=[a,g][g,aa']a.a^{-1}=[g,a]^{-1}[g,aa']$.
So since $\psi$ is a morphism and since $\psi([g,a])=\phi(g,a).a^{-1}$, one gets $\psi(\phi_0(a,[g,a']))=\psi([g,a])^{-1}.\psi([g,aa'])=a(\phi(g,a))^{-1}\phi(g,aa')(aa')^{-1}$ and $a\psi[g,a']a^{-1}=a\phi(g,a')a'^{-1}a^{-1}=a\phi(g,a')(aa')^{-1}$. So these terms are equal if and only if $(\phi(g,a))^{-1}\phi(g,aa')=\phi(g,a')$, or $\phi(g,aa')=\phi(g,a)\phi(g,a')$, i.e. $\phi(g,-)$ is an endomorphism of $A$.\medskip

We now give a  characterization of actions in terms of  extensions to $T(A)$:

\begin{prop}\label{kerxi}
Let $\psi$ be any map $(A|G)\rightarrow A$. Then, using the same notations as above, the following properties are equivalent:

\noindent 1) $\psi$ is an action, i.e. $l_{\psi}$ is a monomorphism;

\noindent 2) $\psi$ satisfies the equivalent conditions of Proposition \ref{extunit} and the arising $\xi:T(A)\rightarrow A$ is such that $\mbox{\em Ker }\xi$ is proper in $A+G$.
\end{prop}

\begin{proof}
\noindent 1) implies 2):
By the proof of Proposition \ref{compar}, if $\psi$ is an action, then it has an extension $\xi:TA\rightarrow A$ which satisfies $\xi.j=\psi$ and $q_{\psi}.k=l_{\psi}.\psi$ (with our usual notations). We claim that Ker $q_{\psi}$ = Ker $\xi$. More precisely, consider ker $q_{\psi}:\mbox{Ker }q_{\psi}\rightarrow A+G$, we claim that ker $q_{\psi}$ factors through $k:TA\rightarrow A+G$ and that the factorization is the kernel of $\xi$. Indeed, in the following diagram, one has $p_G.\mbox{ker }q_{\psi}=0$, so that there exists a unique $m:\mbox{Ker }q_{\psi}\rightarrow TA$ such that $km=\mbox{ker }q_{\psi}$. Then $\xi.m=0$ since $l_{\psi}.\xi.m=0$ and $l_{\psi}$ is a monomorphism. And if $f:X\rightarrow TA$ is such that $\xi.f=0$, then $q_{\psi}.k.f=l_{\psi}.\xi.f=0$, so there exists a unique $x:X\rightarrow \mbox{Ker }q_{\psi}$ such that $\mbox{ker }(q_{\psi}).x=kf$. But this means $kmx= kf$, so $mx=f$ since $k$ is a monomorphism.

$$\xymatrix{
X\ar@{.>}@/^/[rrd]^x\ar@/_/[rdd]_f&&0\ar[d]\\
&\mbox{Ker }q_{\psi}\ar@{=}[r]\ar@{.>}[d]_m&\mbox{Ker }q_{\psi}\ar[d]^{\mbox{\tiny ker }q_{\psi}}\\
0\ar[r]&TA\ar[d]_{\xi}\ar[r]^k&A+G\ar[d]^{q_{\psi}}\ar[r]_{p_G}&G\ar@{=}[d]\ar[r]\ar@/_/[l]_{s_G}&0\\
0\ar[r]&A\ar[r]^{l_{\psi}}&Q_{\psi}\ar[d]\ar[r]_{p_{\psi}}&G\ar[r]\ar@/_/[l]_{s_{\psi}}&0\\
&&0
}$$
Conversely, suppose that $\psi:(A|G)\rightarrow A$ satisfies the conditions of Proposition \ref{extunit}. First note that since $\xi$ is split by $\eta$  one has a split short exact sequence $$\xymatrix{0\ar[r]&\mbox{Ker }\xi\ar[r]^{\mbox{\tiny ker }\xi}&T_GA\ar[r]_{\xi}&A\ar@/_/[l]_{\eta}\ar[r]&0,}$$and in particular $T_GA/\mbox{Ker }\xi=A$. If moreover $\mbox{Ker }\xi$ is  proper in $A+G$, then by Noether's first isomorphism theorem  $T_GA/\mbox{Ker }\xi$ (i.e. $A$) is proper in $A+G/{\mbox{Ker }\xi}$ and $$\frac{\frac{A+G}{\mbox{\tiny Ker }\xi}}{\frac{T_GA}{\mbox{\tiny Ker }\xi}}\cong \frac{A+G}{T_GA}\cong G$$This means that one has the following commutative diagram where all horizontal and vertical lines are short exact sequences:

$$\xymatrix{
&0\ar[d]&0\ar[d]&0\ar[d]\\
0\ar[r]&\mbox{Ker }\xi\ar@{=}[d]\ar[r]^{\mbox{\tiny ker }\xi}&T_GA\ar[d]^k\ar[r]_{\xi}&A\ar[d]^f\ar@/_1pc/[l]_{\eta}\ar[r]&0\\
0\ar[r]&\mbox{Ker }\xi\ar[d]\ar[r]^{k.{\mbox{\tiny ker }\xi}}&A+G\ar[d]^{r_G}\ar[r]^{q'}&\frac{A+G}{\mbox{\tiny Ker } \xi}\ar[d]^g\ar[r]&0\\
0\ar[r]&0\ar[r]\ar[d]&G\ar@{=}[r]\ar[d]&G\ar[d]\ar[r]&0\\
&0&0&0
}$$
But moreover $f=q'k\eta$. Indeed, it suffices to show that $f\xi=q'k\eta\xi$. This is immediate, because  $q'k\eta\xi=f\xi\eta\xi=f\xi$.

But then $q'$ coequalizes $\iota$ and $i_A\psi$. Indeed, $q'.\iota=q'kj=f\xi j=q'k\eta\xi j=q'i_A\psi$. And since $q_{\psi}$ is the coequalizer of $\iota$ and $i_A\psi$, it follows that there exists a unique $h:Q_{\psi}\rightarrow (A+G)/\mbox{Ker }\xi$ such that $hq_{\psi}=q'$. But then $hl_{\psi}=hq_{\psi}i_A=q'i_A=q'k\eta=f$ which is a monomorphism, hence  $l_{\psi}$ is a monomorphism.
\end{proof}

\begin{lemma}\label{conjmult}Consider $\mu_A:T(TA)\rightarrow TA$ the multiplication of the monad and $j_{TA,G}$ the inclusion of $(TA|G)$ in $T(TA)$. Then the composite $\mu_A.j_{TA,G}$ is the restriction to $G$ of the conjugation action of $A+G$ on $TA$ (which exists by Proposition \ref{conjnormsub} since $TA$ is a proper subobject of $A+G$).\end{lemma}
\begin{proof}
By Proposition \ref{stabact}, this restriction is the unique map $c:(TA|G)\rightarrow TA$ making the following diagram commute:
$$\xymatrix{
(TA|G)\ar[d]_c\ar[rr]^{(k_{A,G}|i_G)}&&(A+G|A+G)\ar[d]^{c^2_{A+G}}\\
TA\ar[rr]_{k_{A,G}}&&A+G
}$$

\noindent so we have to show that $c=\mu_A.j_{TA,G}$ makes it commute. This follows from the commutativity of all components of the following diagram:

$$\xymatrix{
(TA|G)\ar[d]_{j_{TA,G}}\ar[rrr]^{(k_{A,G}|i_G)}\ar[rd]^{\iota_{TA,G}}&&&(A+G,A+G)\ar[d]_{\iota_{A+G,A+G}}\ar@/^5pc/[dd]^{c^2_{A+G}}\\
T(TA)\ar[d]_{\mu_A}\ar[r]^{k_{TA,G}}&(TA)+G\ar[rrd]_{<k_{A,G},i_G>}\ar[rr]^{k_{A,G}+i_G}&&(A+G)+(A+G)\ar[d]^{\nabla_{A+G}}\\
TA\ar[rrr]_{k_{A,G}}&&&A+G
}$$

\end{proof}
We denote this restriction  by $c^{TA,A+G}\upharpoonright_G$ or $c$ when no confusion is possible. Note that since it is true in general that the multiplication $\mu_A$ of a monad $\mathbb T$ is always a $\mathbb T$-algebra over $TA$, it is obvious that if the category is semi-abelian and if $\mathbb T$ is $\mathbb T_G$, then $\mu_A$ is a $G$-action over $TA$ in the sense of [BJ], hence $\mu_A.j_{TA,G}$ is an action in our sense: Lemma \ref{conjmult} generalizes this to the non-exact case. It will also be useful in the sequel to note that $c$ factors through $(A|G)$:

\begin{lemma}\label{mufact}Under the hypothesis  of Lemma \ref{conjmult}, the map $c$ factors through the inclusion $j_{A,G}$ of $(A|G)$ in $TA$.\end{lemma}
\begin{proof}Since $j_{A,G}$ is the kernel of $r_A.k_{A,G}$ by Lemma \ref{inclcrosseff}, it suffices to show that $r_A.k_{A,G}.\mu_A.j_{TA,G}=0$. But by definition of $\mu_A$, the right hand square in the following diagram commutes:
$$\xymatrix{
(TA|G)\ar@{.>}[d]\ar[r]^{j_{TA,G}}&TTA\ar[d]^{\mu_A}\ar[r]^{k_{TA,G}}&TA+G\ar[d]^{<k_{A,G},i_G>}\\
(A|G) \ar[r]_{j_{A,G}}&TA\ar[r]_{k_{A,G}}&A+G\ar[r]_{r_G}\ar[d]^{r_A}&G\ar@/_/[l]_{i_G}\\
&&A}$$
so $r_Ak_A\mu j_{TA,G}$ = $r_A<k_A,i_G>k_{TA,G}$ = $<r_Ak_A,r_Ai_G>k_{TA,G}$ = $0.k_{TA,G}$ = 0.
\end{proof}

%\N The resulting factorization $(TA|G)\rightarrow (A|G)$ of $c$ will be denoted by $c'$ in the sequel.\\

\noindent{\bf Example: }In the category of groups, an element of $TA$ is an element of $A+G$ of the form $(\prod_i[g_i,a_i]).a$. We denote by $[-,-]$ the commutators in $A+G$, and by $\llbracket-,-\rrbracket$ the commutators in $(A+G)+G$. Hence an generator of $(TA|G)$ has the form $\llbracket g,(\prod_i[g_i,a_i]).a\rrbracket$ (with $g$ in the second copy of $G$, and $g_i$ in the first one), and $c(\llbracket g,(\prod_i[g_i,a_i]).a\rrbracket)$ = $[g,(\prod_i[g_i.a_i]).a]$ where both commutators are considered in $A+G$. Lemma \ref{mufact} is  illustrated by the fact that this element is  in $(A|G)$. 

\begin{prop}\label{condass}
A map $\xi:TA\rightarrow A$ which satisfies the unit axiom of a $\mathbb T$-algebra also satisfies the associativity axiom  if and only if the following diagram commutes:

$$\xymatrix{
(TA|G)\ar[d]_{(\xi|1_G)}\ar[rrr]^{c^{TA,A+G}\upharpoonright_G}&&&TA\ar[d]^{\xi}\\
(A|G)\ar[rrr]_{\xi.j_{A,G}}&&&A
}$$

\noindent Hence if  a map $\psi:(A|G)\rightarrow A$ satisfies the conditions of Proposition \ref{extunit}, with extension $\xi$ to $TA$, then $(A,\xi)$  is a $\mathbb T$-algebra if and only if the following diagram commutes:
\begin{equation}\label{psiTalgdia}
\xymatrix{
(TA|G)\ar[d]_{({\xi}|1_G)}\ar[rrr]^{c^{TA,A+G}\upharpoonright_G}&&&TA\ar[d]^{\xi}\\
(A|G)\ar[rrr]_{\psi}&&&A
}
\end{equation}

\end{prop}

\begin{proof}
Recall that the associativity condition for $\xi$ is commutativity of the following diagram:
$$\xymatrix{
T(TA)\ar[d]_{T\xi}\ar[r]^{\mu_A}&TA\ar[d]^{\xi}\\
TA\ar[r]_{\xi}&A
}$$

\noindent Since the pair $\{j_{TA,G}:(TA|G)\rightarrow T(TA),\eta_{TA}:TA\rightarrow T(TA)\}$ is (strongly) epimorphic this condition is equivalent to the two equations $\xi.\mu_A.\eta_{TA}=\xi.T\xi.\eta_{TA}$ and $\xi.\mu_A.j_{TA,G}=\xi.T\xi.j_{TA,G}$. The first one is automatically satisfied, because $\xi.T\xi.\eta_{TA}=\xi.\eta_A.\xi=\xi$ since $\eta$ is a natural transformation and $\xi$ satisfies the unit axiom; and $\xi.\mu_A.\eta_{TA}=\xi$ since $\mu_A.\eta_{TA}=1_{TA}$. Considering that in the following diagram the left-hand square commutes, the second condition is equivalent to the commutativity of the outer rectangle; but in view of Lemma \ref{conjmult} this amounts to the required commutativity.
$$\xymatrix{
(TA|G)\ar[r]^{j_{TA,G}}\ar[d]_{(\xi|1_G)}&T(TA)\ar[d]_{T\xi}\ar[r]^{\mu_{A,G}}&TA\ar[d]^{\xi}\\
(A|G)\ar[r]_{j_{A,G}}&TA\ar[r]_{\xi}&A
}$$
\end{proof}

The following result shows that in a semi-abelian category, the conditions of Proposition \ref{condass} are also sufficient for some $\psi:(A|G)\rightarrow A$ to be an action. This provides an alternative proof of the fact that in any semi-abelian category,  all $\mathbb T_G$-algebras are actions. We point out that this proof is based on Proposition \ref{kerxi} and the characterization of proper subobjects by stability under the conjugation action in Corollary \ref{critprop}, instead of Beck's criterion as in \cite{BJ}.

\begin{prop}Let $\mathbb C$ be semi abelian. Let $\psi:(A|G)\rightarrow A$ be a map satisfying the equivalent conditions of Proposition \ref{extunit}. Then $\psi$ is an action if and only if the  diagram (\ref{psiTalgdia}) commutes.

\end{prop}
\begin{proof}The condition is necessary in view of Proposition \ref{condass} (even if the category is not semi-abelian), since if $\psi$ is an action then its extension $\xi$ is an algebra. Conversely, suppose $\mathbb C$ is semi-abelian and suppose that $\psi$ satisfies the conditions of Propositions \ref{extunit} and \ref{condass}. To show that $\psi$ is an action, it suffices by Proposition \ref{kerxi} to show that $\mbox{Ker }\xi$ is proper in $A+G$. Since $\mathbb C$ is semi-abelian, it suffices by Corollary \ref{critprop} to show that $\mbox{Ker }\xi$ is stable under the conjugation action of $A+G$, i.e. that there exists a map $c'':(\mbox{Ker }\xi|A+G)\rightarrow \mbox{Ker }\xi)$ making the following diagram commute:
$$\xymatrix{(\mbox{Ker }\xi|A+G)\ar@{.>}[d]_{c''}\ar[rr]^{({k.\mbox{\tiny ker }\xi}|1)}&&(A+G|A+G)\ar[d]^{c_2^{A+G}}\\
\mbox{Ker }\xi\ar[rr]_{\mbox{\tiny k.ker }\xi}&&A+G}$$
But since $TA$ is proper in $A+G$ one has the conjugation action $c^{TA,A+G}$ of $A+G$ over $TA$ which makes the following diagram commute:
$$\xymatrix{(TA|A+G)\ar[d]_{c^{TA,A+G}}\ar[rr]^{(k|1)}&&(A+G|A+G)\ar[d]^{c_2^{A+G}}\\TA\ar[rr]_k&&A+G}$$
so that one gets the result  if one can find  a map $c''$ such that $$\xymatrix{(\mbox{Ker }\xi|A+G)\ar@{.>}[d]_{c''}\ar[rr]^{({\mbox{\tiny ker }\xi}|1)}&&(TA|A+G)\ar[d]^{c^{TA,A+G}}\\\mbox{Ker }\xi\ar[rr]_{\mbox{\tiny ker }\xi}&&TA}$$commutes, i.e. if $\xi.c^{TA,A+G}.(\mbox{ker }\xi|1)=0$. But
\begin{eqnarray*}
\xi .c^{TA,A+G}(\mbox{ker }\xi|1) &=& \psi.(\xi|1_G).(\mbox{ker }\xi|1)\quad \mbox{by hypothesis}\\
&=& \psi.(\xi.\mbox{ker }\xi|1)\\
&=&\psi.(0|1)\\
&=&0\quad \mbox{by Proposition \ref{elcrprops}.(2).}
\end{eqnarray*}
\end{proof}

We finish with a proposition which characterizes the maps $\psi:(A|G)\rightarrow A$ which can be extended to $\mathbb T$-algebras, hence which are actions if the category is semi-abelian. Of course, the diagrams involved in this proposition are more complicated than the axioms of an $\mathbb T$-algebra. The interest is that these axioms are exclusively expressed in terms of cross effects. Morevover, it appears that the second axiom is a simplification of the associativity axiom: in view of Lemma \ref{conjmult} which shows that $\mu$ is a conjugation action, it can be seen as the associativity axiom restricted to the endofunctor $(-|G)$ instead of $T_G$. And the third axiom expresses that one then must add a condition involving a  higher cross effect. It turns out, however, that in the category of groups the third condition may be skipped.

\begin{prop}\label{3diag}Let $\psi$ be a map $(A|G)\rightarrow A$. Then $\psi$ can be extended to a $\mathbb T$-algebra if and only if the  following three diagrams commute: 
\begin{equation}\label{Talgcond1}
\xymatrix{
((A|G)|A)\ar[rr]^{c^{(A|G),A+G}\upharpoonright_A}\ar[d]_{(\psi|1_A)}&&(A|G)\ar[d]^{\psi}\\
(A|A)\ar[rr]_{c_2^A}&&A
}
\end{equation}

\begin{equation}\label{Talgcond2}
\xymatrix{((A|G)|G)\ar[rr]^{c^{(A|G),A+G}\upharpoonright_G}\ar[d]_{(\psi|1_G)}&&(A|G)\ar[d]^{\psi}\\
(A|G)\ar[rr]_{\psi}&&A}
\end{equation}

\begin{equation}\label{Talgcond3}
\xymatrix{((A|G)|A|G)\ar[rr]^{\iota}\ar[dd]_{\iota}&&((A|G)+A|G)\ar[d]^{(<j_{A,G},\eta_A>|1_G)}\\
&&(TA|G)\ar[d]^{c'}\\
((A|G)+A|G)\ar[d]_{(<\psi,1_A>|1_G)}&&(A|G)\ar[d]^{\psi}\\
(A|G)\ar[rr]_{\psi}&&A}
\end{equation}
where $c'$ is the factorization of $c$  in Lemma \ref{mufact}.
\end{prop}

\begin{proof} Since the first diagram is (\ref{extofpsidia}) we know from  Proposition \ref{extunit} that it commutes iff $\psi$ has an extension $\xi$ which satisfies the unit axiom for an algebra. Supposing that this holds,
it remains to show that then the associativity axiom is equivalent to the commutativity of the two latter diagrams. We apply Proposition \ref{crsdp}  to $F = Id_{\mathbb C}$, $A=G$, $X=TA$, $K=(A|G)$, $Y=A$, $k=j_{A,G}$, $f=r_A.k_{A,G}$, $s=\eta_A$. Then the family $(\iota',(j_{A,G}|1),(\eta_A|1))$ of morphisms with codomain $(TA|G)$ is jointly epimorphic, so diagram (\ref{psiTalgdia}) commutes iff its precompositions with these three maps commute.
To simplify the notation, we will denote by $c$ all conjugation morphisms: naturality of $c$ makes this simplification unambiguous.

First we precompose diagram (\ref{psiTalgdia})  with $(\eta_A|1_G)$:
$$\xymatrix{
(A|G)\ar@/_2pc/@{=}[dd]_{(\xi.\eta_A|1_G)=1_{(A|G)}}\ar@{.>}@/^1pc/[dr]^{j_{A,G}}\ar[d]^{(\eta_A|1_G)}\\ 
(TA|G)\ar[r]^c\ar[d]^{(\xi|1_G)}&TA\ar[d]^{\xi}\\
(A|G)\ar[r]_{\psi}&A}$$
If one shows that the upper triangle commutes, i.e. $c.(\eta_A|1_G)=j_{A,G}$ then one may conclude that the whole outer diagram commutes, so that the resulting  condition is void here. In fact, since $k_{A,G}$ is a monomorphism, it suffices to show that $k_{A,G}.c.(\eta_A|1_G)=k_{A,G}.j_{A,G}=\iota_{A,G}$. But by  Proposition \ref{stabact}, one has $k_{A,G}.c=c^{A+G}_2.(k_{A,G}|i_G)$, so one has:\\

\noindent$k_{A,G}.c.(\eta_A|1_G)$

$=c^{A+G}_2.(k_{A,G}|i_G).(\eta_A|1_G)$

$=c^{A+G}_2.(k_{A,G}.\eta_A|i_G)$

$=c^{A+G}_2.(i_A|i_G)$

$=\iota_{A,G}$ because the following diagram commutes:
$$\xymatrix{
(A|G)\ar[rr]^{(i_A|i_G)}\ar[d]_{\iota_{A,G}}&&(A+G|A+G)\ar[d]^{\iota_{A+G,A+G}}\\
A+G\ar@{=}[drr]\ar[rr]^{i_A+i_G}&&(A+G)+(A+G)\ar[d]^{\nabla_{A+G}}\\
&&A+G}$$
Secondly, we we precompose diagram (\ref{psiTalgdia})  with $(j_{A,G}|1)$:

$$\xymatrix{
((A|G)|G)\ar@/_2.5pc/[dd]_{(\xi.j_{A,G}|1_G)=(\psi|1_G)}\ar[r]^c\ar[d]^{(j_{A,G}|1_G)}&(A|G)\ar[d]^{j_{A,G}}\ar@/^2.5pc/[dd]^{\psi}\\ 
(TA|G)\ar[r]^c\ar[d]^{(\xi|1_G)}&TA\ar[d]^{\xi}\\
(A|G)\ar[r]_{\psi}&A}$$%
All parts of this diagram except the bottom square are known to commute, so diagram (\ref{psiTalgdia})  precomposed with $(j_{A,G}|1)$ commutes iff diagram (\ref{Talgcond2}) commutes.

Finally, we precompose diagram (\ref{psiTalgdia})  with  $\iota'=(<j_{A,G},\eta_A>|1_G).\iota$, $\iota$ here  being the injection of $((A|G)|A|G)$ into $((A|G)+A|G)$: 

$$\xymatrix{((A|G)|A|G)\ar[d]_{\iota}\\
((A|G)+A|G)\ar[rrr]^-{c'(<j_{A,G},\eta_A>|1_G)}\ar@/_2pc/[dd]_{(<\psi,1_A>|1_G)}\ar[d]^{(<j_{A,G},\eta_A>|1_G)}&&&(A|G)\ar[d]^{j_{A,G}}\\
(TA|G)\ar[d]^{(\xi|1_G)}\ar[rrr]^c&&&TA\ar[d]^{\xi}\\
(A|G)\ar[rrr]_{\psi}&&&A}$$%
The left-hand triangle commutes since $(\xi|1)(<j_{A,G},\eta_A>|1_G)$ = $(<\xi j_{A,G},\xi\eta_A>|1_G)$ = $(<\psi,1_A>|1_G)$, and  the upper square commutes by definition of $c'$.
Hence diagram (\ref{psiTalgdia})  precomposed with $\iota'$ commutes iff diagram (\ref{Talgcond3}) commutes, which achieves the proof.
\end{proof}

\noindent{\bf Example: }In the category of groups, consider a map $\psi:(A|G)\rightarrow A$ and the corresponding $\phi:G\times A\rightarrow G$\footnote{Note that if one defines $\phi$ from $\psi$, one can put $\phi(g,a)=\psi([g,a]).a$, even if $[g,a]$ is the unit, i.e. $g$ is the unit of $G$ or $a$ is the unit of $A$: this insures that $\phi(1,a)=a$ and $\phi(g,1)=1$, and of course $\psi[g,a]=\phi(g,a).a^{-1}$. So there is no need for a special discussion for the case when $gg_1=1$ in what follows.}. We know that the first axiom in Proposition \ref{3diag} expresses that the $\phi(g,-)$ are group endomorphisms of $A$. We now examine the second diagram. Considering that $(X|Y)$ is the subgroup of $Y+X$ generated by the $[x,y]$'s (with $x$, $y$ different from the units), a generator of $((A|G)|G)$ has the form $\llbracket g,\prod_{i= 1}^n[g_i,a_i]^{z_i}\rrbracket$, where the outer commutator $\llbracket\ldots\rrbracket$ is considered in $(A+G)+G$, while the inner one $[\ldots]$ is in $A+G$. First consider the case  $n=1$ with $z_1=1$. For a generator $\llbracket g,[g_1,a_1]\rrbracket$ one has $c(\llbracket g,[g_1,a_1]\rrbracket)=[g,[g_1,a_1]]$ = $[gg_1,a_1][a_1,g][a_1,g_1]$ = $[gg_1,a_1][g,a_1]^{-1}[g_1,a_1]^{-1}$, the second equality arising from the formula $[x,[y,z]]=[xy,z][z,x][z,y]$. So one gets $\psi(c(\llbracket g,[g_1,a_1]\rrbracket)$ = $\psi([gg_1,a_1][g,a_1]^{-1}[g_1,a_1]^{-1})$ = $\psi([gg_1,a_1])(\psi([g,a_1])^{-1}\psi([g_1,a_1])^{-1}$ = $\phi(gg_1,a_1)a_1^{-1}.(\phi(g,a_1).a_1^{-1})^{-1}.(\phi(g_1,a_1).a_1^{-1})^{-1}$ = $\phi(gg_1,a_1).(\phi(g,a_1))^{-1}.a_1.(\phi(g_1,a_1))^{-1}$ on the one hand. On the other hand, one gets $\psi((\psi|1_G))(\llbracket g,[g_1,a_1]\rrbracket)$ = $\psi([g,\psi([g_1,a_1]))$ = $\psi([g,\phi(g,a_1)a_1^{-1}])$ = $\phi(g,\phi(g,a_1)a_1^{-1})(\phi(g,a_1)a_1^{-1})^{-1}$ =  $\phi(g,\phi(g,a_1)a_1^{-1}))a_1(\phi(g,a_1))^{-1}$. %
So the equation $\psi(c(\llbracket g,[g_1,a_1]\rrbracket)$ = $\psi((\psi|1_G))(\llbracket g,[g_1,a_1]\rrbracket)$ is verified if and only if so is the equation $\phi(gg_1,a_1).(\phi(g,a_1))^{-1}$ = $\phi(g,\phi(g_1,a_1).a_1^{-1})$. Now if one assumes moreover that  the first diagram commutes, i.e. that the $\phi(g,-)$'s are endomorphisms of $A$, then this amouts to the relation $\phi(gg_1,a_1).(\phi(g,a_1))^{-1}$ = $\phi(g,\phi(g_1,a_1)).(\phi(g,a_1))^{-1}$, i.e. $\phi(gg_1,a_1)=\phi(g,\phi(g_1,a_1))$. So, assuming the commutativity of the first diagram in Proposition \ref{3diag}, the commutativity of the second one  implies that $\phi$ is a group action in the usual sense; and of course, conversely, if $\phi$ is an action in the usual sense, then the corresponding $\psi$ is an action in our sense, hence the three diagrams of Proposition \ref{3diag} commute.\\

%%%%%%%%%%%%%%%%

\end{document}